\documentclass{amsart}
\usepackage[english]{babel}
\usepackage{amsmath}
\usepackage{mathrsfs}
\usepackage{amsthm}
\usepackage[all]{xy}
\usepackage{amsfonts}
\usepackage{amssymb}
\usepackage{hyperref}

\usepackage{commath}
\usepackage{mathtools}

\newtheorem{teo}{Theorem}[section]
\newtheorem{coro}[teo]{Corollary}
\newtheorem{lema}[teo]{Lemma}
\newtheorem{propo}[teo]{Proposition}

\theoremstyle{definition}
\newtheorem{defi}[teo]{Definition}

\newtheorem{remark}[teo]{Remark}

\numberwithin{equation}{section}

\begin{document}

\title[Toeplitz and moment maps on domains of type III]{Commuting Toeplitz operators and moment maps on Cartan domains of type III}

\author[Cuevas-Estrada]{David Cuevas-Estrada}
\address{Centro de Investigaci\'on en Matem\'aticas\\
	Guanajuato, Guanajuato\\
	M\'exico}
\email{david.cuevas@cimat.mx}

\author[Quiroga-Barranco]{Raul Quiroga-Barranco}
\address{Centro de Investigaci\'on en Matem\'aticas\\
	Guanajuato, Guanajuato\\
	M\'exico}
\email{quiroga@cimat.mx}

\subjclass{Primary 47B35 30H20; Secondary 53D20}

\keywords{Toeplitz operators, Bergman spaces, Cartan domains, Lie groups, K\"ahler manifolds, moment maps}


\begin{abstract}
	Let $D^{III}_n$ and $\mathscr{S}_n$ be the Cartan domains of type III that consist of the symmetric $n \times n$ complex matrices $Z$ that satisfy $Z\overline{Z} < I_n$ and $\mathrm{Im}(Z) > 0$, respectively. For these domains, we study weighted Bergman spaces and Toeplitz operators acting on them. We consider the Abelian groups $\mathbb{T}$, $\mathbb{R}_+$ and $\mathrm{Symm}(n,\mathbb{R})$ (symmetric $n \times n$ real matrices), and their actions on the Cartan domains of type III. We call the corresponding actions Abelian Elliptic, Abelian Hyperbolic and Parabolic. The moment maps of these three actions are computed and functions of them (moment map symbols) are used to construct commutative $C^*$-algebras generated by Toeplitz operators. This leads to a natural generalization of known results for the unit disk. We also compute spectral integral formulas for the Toeplitz operators corresponding to the Abelian Elliptic and Parabolic cases.
\end{abstract}

\maketitle

\section{Introduction}
Bounded symmetric domains, weighted Bergman spaces on such domains and Toeplitz operators acting on Bergman spaces constitute three fundamental objects in operator theory. The reason is that they are specific enough to make explicit computations that lead to interesting results, and at the same time they are complicated enough so that such results are non-trivial and enlightening.

For some years now, operator theory analysts have found plenty of examples of commutative $C^*$-algebras generated by Toeplitz operators when the corresponding set of symbols is suitably restricted. The first such example was considered in \cite{KorenblumZhu1995}, where it was proved that Toeplitz operators on the unit disk $\mathbb{D}$ with radial symbols are diagonal with respect to the orthogonal monomial basis. Clearly, a symbol on $\mathbb{D}$ is radial if it is invariant under the natural $\mathbb{T}$-action. We note that the $\mathbb{T}$-action on the unit disk $\mathbb{D}$ realizes, up to conjugacy, all the elliptic M\"obius transformations.

The introduction in \cite{KorenblumZhu1995} of Toeplitz operators with radial symbols was followed by a series of developments found in \cite{GKVElliptic,GKVHyperbolic,GKVParabolic}. These references considered all three fundamental types of M\"obius transformations on the unit disk $\mathbb{D}$: elliptic, hyperbolic and parabolic. It was proved that symbols that are invariant under the corresponding groups of M\"obius transformations yield Toeplitz operators that generate commutative $C^*$-algebras. Then, it was found in \cite{GQVJFA} that, under suitable smoothness conditions, these constructions yield the only commutative $C^*$-algebras generated by Toeplitz operators acting on every weighted Bergman space on the unit disk $\mathbb{D}$.

The next step was to study the behavior in the case of higher dimensional bounded symmetric domains, and the unit ball $\mathbb{B}^n$ in $\mathbb{C}^n$ was the first natural example to consider. It was found in \cite{QVBallI,QVBallII} that there exists exactly, up to conjugacy, $n+2$ maximal Abelian subgroups (MASGs) of biholomorphisms each one of which yields invariant symbols whose Toeplitz operators generate commutative $C^*$-algebras. This is a natural generalization of the situation observed for the unit disk $\mathbb{D}$, since in this case we have $n=1$ from which it follows the existence of three MASGs. Nevertheless, some simplicity is lost because the number of MASGs grows with the dimension of the unit ball $\mathbb{B}^n$.

After these works, many other results have been found where a suitable symmetry of the symbols yields commuting Toeplitz operators. Such symmetry is in most cases a consequence of the invariance with respect to a certain biholomorphism group. This has been observed for every bounded symmetric domain on every weighted Bergman space. We refer to \cite{DOQJFA} for a very general collection of related results.

In a parallel line of development, symplectic geometry has been found to play an special role in the construction of symbols whose Toeplitz operators generate commutative $C^*$-algebras. It was proved in \cite{QSJFAUnitBall} that, for the unit ball $\mathbb{B}^n$ and on any of its weighted Bergman spaces, every single Abelian connected group of biholomorphisms provides symbols with mutually commuting Toeplitz operators. For such a group $H$ acting on $\mathbb{B}^n$ this is achieved by considering the so-called moment map symbols for $H$ instead of $H$-invariant symbols. We refer to Section~\ref{sec:CartanIIIGeometry} for the details of the definitions and properties involved. However, we mention here that the moment map of an action is a mapping defined on the corresponding bounded symmetric domain using its symplectic manifold structure, and the moment map symbols are functions of such moment maps. Another example of the use of moment map symbols is given by the results found in \cite{QRCAOT}, where the bounded symmetric domain considered is the Cartan domain of type IV.

The goal of this work is to apply these ideas to study Toeplitz operators with moment map symbols acting on the weighted Bergman spaces of Cartan domains of type III. We recall that such domains are realized by the so-called generalized unit disk $D^{III}_n$ and Siegel's generalized upper half-plane $\mathscr{S}_n$ (see Section~\ref{sec:CartanIIIAnalysis}). In fact, as we show in Section~\ref{sec:3Groups}, to these domains we can associate three biholomorphic actions that naturally generalize the three actions described above for the unit disk. Hence, we call these actions on either $D^{III}_n$ or $\mathscr{S}_n$ the Elliptic, Hyperbolic and Parabolic Actions (see subsection~\ref{subsec:ellhyperpar_actions}). These come from the groups $\mathrm{U}(n)$, $\mathrm{GL}(n,\mathbb{R})$ and $\mathrm{Symm}(n,\mathbb{R})$, respectively, of which only the last one is Abelian for every $n$. Hence, we introduce actions that we call Abelian Elliptic and Abelian Hyperbolic (see Definition~\ref{defi:AbelianEllHyp}). As noted in Remark~\ref{rmk:AbelianActions} all three Abelian actions can be seen as coming from the corresponding centers of the original groups involved.

We present in Section~\ref{sec:CartanIIIAnalysis} all the theory needed to understand the Riemannian and symplectic geometry background used in the rest of the work. In particular, we compute in subsection~\ref{subsec:BergmanmetricKahler} the Bergman metric and the K\"ahler form for both $D^{III}_n$ and $\mathscr{S}_n$. We use this to compute in subsection~\ref{subsec:momentmaps_Abelian} the moment maps for our three distinguished actions: Abelian Elliptic, Abelian Hyperbolic and Parabolic.

We introduce in Section~\ref{sec:ToeplitzSpecialSymbols} Toeplitz operators with special symbols. First, we consider invariant symbols in subsection~\ref{subsec:inv-symbols} and we recall some known commutative $C^*$-algebras generated by Toeplitz operators for our setup. Second, we introduce moment map symbols in Definition~\ref{defi:momentmap_symbol}, and with the use of our moment map computations we obtain explicit formulas for moment map symbols for our three distinguished Abelian actions. We obtain the following general description (see Proposition~\ref{propo:moment_symbols_explicit} for the precise statements)
\begin{itemize}
	\item Abelian Elliptic symbols:
		$Z \longmapsto f\big(\mathrm{tr}\big((I_n - Z\overline{Z})^{-1}\big)\big)$.
	\item Abelian Hyperbolic symbols:
		$Z \longmapsto 
		f\big(\mathrm{tr}(\mathrm{Im}(Z)^{-1} \mathrm{Re}(Z))
		\big)$.
	\item Parabolic symbols:
		$Z \longmapsto f(\mathrm{Im}(Z))$.
\end{itemize}
From a quick comparison with the notions considered in the current literature, we observe that these three types of symbols are natural, almost canonical, generalizations from the unit disk $\mathbb{D}$ to the domains $D^{III}_n$ and $\mathscr{S}_n$ of the symbols obtained from the elliptic, hyperbolic and parabolic actions on~$\mathbb{D}$.

We prove in Theorem~\ref{teo:commToep_3Abeliangroups} that the three types of symbols above yield Toeplitz operators that generate commutative $C^*$-algebras on every weighted Bergman space. Our method of proof is based on the fact that these moment map symbols have an additional invariance: they are invariant under the group from which the Abelian group is the center. This allows to use the results from subsection~\ref{subsec:inv-symbols}. On the other hand, it is interesting to observe the importance of having only three types of symbols in Theorem~\ref{teo:commToep_3Abeliangroups} as a generalization of the corresponding result for the unit disk. This is explained in Remark~\ref{rmk:3Abelian_DIII_vs_Disk}.

Finally, we obtain in Section~\ref{sec:SpectralIntegrals} integral formulas for the Toeplitz operators with moment map symbols that provide simultaneous diagonalization for them. This is done for the Abelian Elliptic and Parabolic Actions; we leave the Abelian Hyperbolic case as an important project to develop. The relevant results are Theorems~\ref{teo:Toeplitz_muT} and \ref{teo:Toeplitz_SymmnR}. The simplicity of the formulas presented in Theorem~\ref{teo:Toeplitz_muT} highlights the importance of using symplectic geometry to solve these operator theory problems. Likewise, Theorem~\ref{teo:Toeplitz_SymmnR} has very natural formulas that involve a Fourier-Laplace transform obtained in Theorem~\ref{teo:RlambdaFourierLaplace}.

\section{The Cartan domains of type III and their analysis}
\label{sec:CartanIIIAnalysis}
We recall the basic geometric and analytic properties of the Cartan domains of type III.

\subsection{Bounded and unbounded realizations}
In the rest of this work, and for $\mathbb{F}$ either $\mathbb{R}$ or $\mathbb{C}$, we will denote by $\mathrm{Mat}(n,\mathbb{F})$ the space of $n \times n$ matrices over $\mathbb{F}$ and by $\mathrm{Symm}(n,\mathbb{F})$ its subspace of symmetric matrices. As usual, $\mathrm{GL}(n,\mathbb{F})$ will denote the Lie group of invertible elements of $\mathrm{Mat}(n,\mathbb{F})$.

\begin{defi}
	The $n$-th Cartan domain of type III is the complex domain given by $D^{III}_n=\{Z\in \mathrm{Symm}(n,\mathbb{C}) \mid I_n-Z\overline{Z} >0\}$.
\end{defi}

The domain $D^{III}_n$ is clearly bounded. On the other hand, there is a natural unbounded domain associated to $D^{III}_n$.

\begin{defi}
	The $n$-th generalized Siegel domain is the complex domain given by $\mathscr{S}_n = \{ Z \in \mathrm{Symm}(n,\mathbb{C}) \mid \mathrm{Im}(Z) > 0 \}$.
\end{defi}

We note that $D^{III}_1$ and $\mathscr{S}_1$ are precisely the unit disk $\mathbb{D}$ and the upper half-plane $\mathbb{H}$, respectively, in the complex plane $\mathbb{C}$. For this reason, the domains $D^{III}_n$ and $\mathscr{S}_n$ are also known as the generalized unit disk and generalized upper-half plane, respectively. Furthermore, these domains are related in a way similar to the well known $1$-dimensional case. For the next result we refer to \cite[Exercise~C,~Chapter~VIII]{Helgason}.

\begin{propo}\label{propo:diffoftwodomains}
	The map $\varphi : \mathscr{S}_n \rightarrow D^{III}_n$ given by
	\[
		Z\mapsto (I_n+iZ)(I_n-iZ)^{-1},
	\]
	is a biholomorphism from $\mathscr{S}_n$ onto $D^{III}_n$.
\end{propo}

Because of the previous result, the domain $\mathscr{S}_n$ is also known as the unbounded realization of the $n$-th Cartan domain of type III.

\subsection{Biholomorphism groups}
In this section we describe the groups of biholomorphisms of the domains $D^{III}_n$ and $\mathscr{S}_n$ introduced above. We start by considering the matrices
\[
	I_{n,n} =
		\begin{pmatrix}
			I_n & 0 \\ 
			0 & -I_n
		\end{pmatrix}, \quad
	J_n = 
		\begin{pmatrix}
			0 & -I_n \\ 
			I_n & 0
		\end{pmatrix}.
\]
These naturally yield the next Lie groups.
\begin{align*}
	\mathrm{Sp}(n,\mathbb{C}) 
		&= \{ M \in \mathrm{Mat}(2n,\mathbb{C}) 
			\mid M^\top J_n M = J_n \}, \\
	\mathrm{Sp}(n,\mathbb{R}) 
		&= \{ M \in \mathrm{Mat}(2n,\mathbb{R}) 
			\mid M^\top J_n M = J_n \}, \\
	\mathrm{U}(n,n) 
		&= \{ M \in \mathrm{Mat}(2n,\mathbb{C}) 
			\mid M^* I_{n,n} M = I_{n,n} \}.
\end{align*}
We recall the notion of a bounded symmetric domain.
\begin{defi}
	A domain $D \subset \mathbb{C}^N$ is called symmetric if for every $z \in D$ there is a biholomorphism $\varphi_z : D \rightarrow D$ that satisfies $\varphi_z(w) = w$ if and only if $w = z$. If $D$ is also bounded, then $D$ is called a bounded symmetric domain. If $D$ satisfies $tD = D$, for every $t \in \mathbb{T}$, then the domain $D$ is called circled.
\end{defi}

Through suitable actions of the groups introduced above, one can prove that the domains $D^{III}_n$ and $\mathscr{S}_n$ are symmetric. For the next result we refer to \cite[Paragraph~(2.3)]{Mok} (see also \cite{Helgason}). From now on, for any given matrix $M \in \mathrm{Mat}(2n,\mathbb{C})$ a decomposition of the form
\[
	M = 
	\begin{pmatrix}
		A & B \\
		C & D
	\end{pmatrix},
\]
will always be taken so that $A,B,C,D$ have size $n \times n$.

\begin{propo}\label{propo:biholoDIIIn}
	The action via generalized M\"obius transformations given~by
	\begin{align*}
		\mathrm{Sp}(n,\mathbb{C})\cap \mathrm{U}(n,n)
			\times D^{III}_n & \longrightarrow D^{III}_n \\
		\begin{pmatrix}
				A & B \\ 
				C & D		
		\end{pmatrix}\cdot Z & \longmapsto (AZ+B)(CZ+D)^{-1},
	\end{align*}
	realizes the biholomorphism group of $D^{III}_n$. Furthermore, $D^{III}_n$ is a circled bounded symmetric domain and it is given as the quotient
	\[
		D^{III}_n \simeq 
			\mathrm{Sp}(n,\mathbb{C}) \cap \mathrm{U}(n,n) /
				\mathrm{U}(n),
	\]
	where $\mathrm{U}(n)$ embedded in $\mathrm{Sp}(n,\mathbb{C}) \cap \mathrm{U}(n,n)$ by
	\[
		A \longmapsto 
			\begin{pmatrix}
				A & 0 \\
				0 & \overline{A}
			\end{pmatrix}
	\]
	corresponds to the group of biholomorphisms of $D^{III}_n$ that fix the origin.
\end{propo}

Similarly, we have the next description of the biholomorphism group of the domain $\mathscr{S}_n$. We now refer to \cite[Exercise~C,~Chapter~VIII]{Helgason}.

\begin{propo}\label{propo:biholoSn}
	The action via generalized M\"obius transformations given by
	\begin{align*}
		\mathrm{Sp}(n,\mathbb{R})
			\times \mathscr{S}_n & \longrightarrow \mathscr{S}_n \\
			\begin{pmatrix}
				A & B \\ 
				C & D		
			\end{pmatrix}\cdot Z & \longmapsto (AZ+B)(CZ+D)^{-1},
	\end{align*}
	realizes the biholomorphism group of $\mathscr{S}_n$. Furthermore, $\mathscr{S}_n$ is a symmetric domain and it is given as the quotient
	\[
		\mathscr{S}_n = \mathrm{Sp}(n,\mathbb{R})/\mathrm{U}(n),
	\]
	where $\mathrm{U}(n)$ embedded in $\mathrm{Sp}(n,\mathbb{R})$ by
	\[
		A \longmapsto
			\begin{pmatrix}
				\mathrm{Re}(A) & \mathrm{Im}(A) \\
				-\mathrm{Im}(A) & \mathrm{Re}(A)
			\end{pmatrix}
	\]
	corresponds to the group of biholomorphisms of $\mathscr{S}_n$ that fix the matrix $iI_n$.
\end{propo}

\begin{remark}
	By Proposition~\ref{propo:diffoftwodomains} it follows that the biholomorphism groups of $D^{III}_n$ and $\mathscr{S}_n$ are isomorphic. In fact, it is easy to prove that $\mathrm{Sp}(n,\mathbb{C}) \cap \mathrm{U}(n,n)$ and $\mathrm{Sp}(n,\mathbb{R})$ are conjugated (see \cite{Mok}).
\end{remark}

\subsection{Bergman spaces and Toeplitz operators}
\label{subsec:BergmanToeplitz}
From now on, $D$ will denote either of the domains $D^{III}_n$ or $\mathscr{S}_n$, and $\dif Z$ the Lebesgue measure on $\mathrm{Symm}(n,\mathbb{C})$. A number of invariants can be associated to any symmetric domain. The simplest one is the dimension, which for $D$ is $n(n+1)/2$. For other invariants we refer to \cite{Upmeier} for further details on their definitions and here we simply state their known values with some remarks.

\begin{itemize}
	\item The rank is defined as the dimension of maximal linearly embedded poly-disks. For $D$ the rank is $n$.
	\item The multiplicities are defined as the main invariants that describe the Jordan triple system associated to the symmetric domain. For $D$ the multiplicities are $a = 1$, $b = 0$. The vanishing of the latter implies that $D^{III}_n$ has a tubular realization which is in fact given by $\mathscr{S}_n$. For this we observe that
	\[
		\mathscr{S}_n = \mathrm{Symm}(n,\mathbb{R}) 
				\oplus i \mathrm{Pos}(n,\mathbb{R}),
	\]
	where $\mathrm{Pos}(n,\mathbb{R})$ denotes the cone of positive definite $n \times n$ real matrices. In the rest of this work we will denote $\Omega_n = \mathrm{Pos}(n,\mathbb{R})$.
	\item For a tubular domain, the genus is given as $p = 2d/r$, where $d$ and $r$ are the dimension and the rank of the domain, respectively. Hence, for $D$ we have $p = n+1$.
\end{itemize}

We will make use of the multi-gamma function (see \cite[Definition~2.4.2]{Upmeier}) that we will consider only for Cartan domains of type III. Such function is associated to the cone part of a tubular realization of a tube type symmetric domain. In our case, it is defined by
\[
	\Gamma_{\Omega_n}(\lambda) = 
		(2\pi)^{\frac{n(n-1)}{4}}
			\prod_{j=1}^n \Gamma\bigg(\lambda-\frac{j-1}{2}\bigg),
\]
for every $\lambda > (n-1)/2$. It is well known (see \cite{Hua,Upmeier}) that the volume of a bounded symmetric domain can be expressed in terms of the multi-gamma functions. In this case we have (see \cite{Hua})
\[
	\mathrm{Vol}(D^{III}_n) = 	
		 \frac{\pi^{\frac{n(n+1)}{2}} 
		 	\Gamma_{\Omega_n}\big(\frac{n+1}{2}\big)}%
			{\Gamma_{\Omega_n}(n+1)}.
\]
Hence, we consider the normalized measure on $\mathrm{Symm}(n,\mathbb{C})$
\[
	\dif v(Z) = 		
			\frac{\Gamma_{\Omega_n}(n+1)}%
			{\pi^{\frac{n(n+1)}{2}} \Gamma_{\Omega_n}\big(\frac{n+1}{2}\big)} \dif Z.
\]
In particular, $\dif v(Z)$ is a probability measure on $D^{III}_n$.

\begin{defi}
	The (weightless) Bergman space $\mathcal{A}^2(D)$ is the subspace of $L^2(D,v)$ that consists of holomorphic functions. In other words, we have
	\[
		\mathcal{A}^2(D)=\{f\in L^2(D,v) \mid f 
			\text{ is holomorphic }\}.
	\]
\end{defi}
It is a well known fact that $\mathcal{A}^2(D)$ is a closed subspace of $L^2(D,v)$ (see~\cite{Helgason,Upmeier}). We will denote by $B_D : L^2(D,v) \rightarrow \mathcal{A}^2(D)$ the corresponding orthogonal projection. It is called the (weightless) Bergman projection. Moreover, it is also well known that $\mathcal{A}^2(D)$ is a reproducing kernel Hilbert space (see~\cite[Chapter~VIII]{Helgason}) in the sense that the evaluation map
\begin{align*}
	\mathrm{ev}_Z : \mathcal{A}^2(D) & \longrightarrow \mathbb{C}\\ 	
	f & \longmapsto f(Z),
\end{align*}
is continuous for every  $Z\in D$. This implies the existence of a unique smooth function $K_D : D \times D \rightarrow \mathbb{C}$, holomorphic in the first variable and anti-holomorphic in the second variable, satisfying $\overline{K_D(Z,\cdot)} \in \mathcal{A}^2(D)$ for every $Z \in D$ and for which the Bergman projection is given by
\[
	B_D(f)(Z) = \int_D f(W) K_D(Z,W) \dif v(W).
\]
for every $f \in L^2(D,v)$ and $Z \in D$. The function $K_D$ is called the (weightless) Bergman kernel of $D$.

The Bergman kernels of symmetric domains have closed known expressions. In particular, it follows from Examples~2.4.17 and 2.9.15 in \cite{Upmeier} that the Bergman kernels of $D^{III}_n$ and $\mathscr{S}_n$ are given by the expressions
\begin{align}
	K_{D^{III}_n}(Z,W) &= \det(I_n-Z\overline{W})^{-(n+1)}, 
	\label{eq:KDIIIn} \\
	K_{\mathscr{S}_n}(Z,W) & =\det(-i(Z-\overline{W}))^{-(n+1)},
	\label{eq:KSn}
\end{align}
respectively. We note that a linear biholomorphism has to be applied in order to obtain the above expression for $K_{\mathscr{S}_n}$ from the one found in \cite{Upmeier}. More precisely, our unbounded realization of $D^{III}_n$ is obtained from the one considered in \cite{Upmeier} through the map $Z\mapsto -iZ$.

The next standard construction is to use powers of the Bergman kernel to obtain weighted measures. The following formula, which holds for every $\lambda > n$, is useful to normalize such weighted measures (see \cite[Lemma~2.9.18]{Upmeier})
\[
	\int_{D^{III}_n} \det(I_n-Z\overline{Z})^{\lambda-n-1} \dif Z
	 = \frac{\pi^{\frac{n(n+1)}{2}} 	
	 	\Gamma_{\Omega_n}\big(\lambda-\frac{n+1}{2}\big)}%
	 		{\Gamma_{\Omega_n}(\lambda)}.
\]
Hence, we consider for every $\lambda > n$ the measure
\[
	\dif v_\lambda(Z) =
		\frac{\Gamma_{\Omega_n}\left(\lambda\right)}%
			{\pi^{\frac{n(n+1)}{2}} \Gamma_{\Omega_n}\big(\lambda-\frac{n+1}{2}\big)}
				\det(I_n-Z\overline{Z})^{\lambda-n-1} \dif Z
\]
which is a probability measure on $D^{III}_n$, and we also consider the normalized measure
\[
	\dif \widehat{v}_\lambda(Z) =
		\frac{\Gamma_{\Omega_n}\left(\lambda\right)}%
			{\pi^{\frac{n(n+1)}{2}} \Gamma_{\Omega_n}\left(\lambda-\frac{n+1}{2}\right)}
				\det(-i(Z-\overline{Z}))^{\lambda-n-1} \dif Z.
\]
on the domain $\mathscr{S}_n$. 

\begin{defi}
	For $\lambda > n$, the weighted Bergman spaces on $D^{III}_n$ and $\mathscr{S}_n$ with weight $\lambda$ are given by
	\begin{align*}
	\mathcal{A}^2_\lambda(D^{III}_n) &= 
		\{f\in L^2(D^{III}_n,v_\lambda) \mid f 
			\text{ is holomorphic }\}, \\
	\mathcal{A}^2_\lambda(\mathscr{S}_n) &=
		\{f\in L^2(\mathscr{S}_n,\widehat{v}_\lambda) \mid 
			f \text{ is holomorphic } \},
	\end{align*}
	respectively. We will denote by $\mathcal{A}^2_\lambda(D)$ the corresponding weighted Bergman space when $D$ is $D^{III}_n$ or $\mathscr{S}_n$.
\end{defi}
Note that for $\lambda = n+1$, we obtain  $\mathcal{A}^2_{n+1}(D^{III}_n)=\mathcal{A}^2(D^{III}_n)$ and $\mathcal{A}^2_{n+1}(\mathscr{S}_n)=\mathcal{A}^2(\mathscr{S}_n)$, which are the weightless Bergman spaces.

As before, it is well known that every weighted Bergman space is closed in the corresponding $L^2$ space in such a way that it is a reproducing kernel Hilbert space. In particular, for $D$ either $D^{III}_n$ or $\mathscr{S}_n$ there exists a smooth function $K_{D,\lambda} : D \times D \rightarrow \mathbb{C}$, holomorphic and anti-holomorphic in the first and second variable (respectively), such that the orthogonal projection onto $\mathcal{A}^2_\lambda(D)$ is given by
\[
	B_{D,\lambda}(f)(Z) = \int_D f(W) K_{D,\lambda}(Z,W)
			\dif \nu_\lambda(W),
\]
for every $f \in L^2(D,\nu_\lambda)$ and $Z \in D$, where $\nu_\lambda$ denotes either $v_\lambda$ or $\widehat{v}_\lambda$ according to whether $D$ is $D^{III}_n$ or $\mathscr{S}_n$. This projection is called the weighted Bergman projection. It follows by Propositions~2.4.22 and 2.9.24 from \cite{Upmeier} that the weighted Bergman kernels for these domains are given by the following expressions
\begin{align*}
	K_{D^{III}_n,\lambda}(Z,W) 
		&= \det(I_n-Z\overline{W})^{-\lambda}, \\
	K_{\mathscr{S}_n,\lambda}(Z,W) 		
		&=\det(-i(Z-\overline{W}))^{-\lambda},
\end{align*}
for every $\lambda > n$. In particular, we have $K_{D,\lambda}(Z,W) = K_D(Z,W)^\frac{\lambda}{n+1}$ for every $Z,W \in D$.

The previous constructions allow us to define our main object of study.

\begin{defi}
	For every weight $\lambda > n$ and $a \in L^\infty(D)$, the Toeplitz operator with symbol $a$ is the bounded operator $T^{(\lambda)}_a$ acting on $\mathcal{A}^2_\lambda(D)$ that is given by $T^{(\lambda)}_a = B_{D,\lambda}\circ M_a$.
\end{defi}

It is interesting to note that the Bergman spaces $\mathcal{A}^2_\lambda(D^{III}_n)$ and $\mathcal{A}^2_\lambda(\mathscr{S}_n)$ are unitarily equivalent, thus simplifying some computations. This unitary equivalence is stated without proof in the next result. Its proof is a straightforward generalization of the arguments provided to obtain Theorem~4.9 in Chapter~IV from \cite{Range}.

\begin{teo}
	The map $\varphi$ given in Proposition~\ref{propo:diffoftwodomains} induces the unitary operator given by
	\begin{align*}
		U_\varphi:\mathcal{A}^2_\lambda(D^{III}_n) &\longrightarrow 
			\mathcal{A}^2_\lambda(\mathscr{S}_n)\\
		f &\longmapsto 
			J_\mathbb{C} (\varphi)^{\frac{\lambda}{n+1}} 	f\circ\varphi,
	\end{align*}
	where $J_\mathbb{C}(\varphi) = \det(\dif\varphi_\mathbb{C})$ denotes the complex Jacobian.
\end{teo}

\section{Geometry of Cartan domains of type III}
\label{sec:CartanIIIGeometry}

\subsection{Symplectic and Kähler geometry}
We discuss here some basic material from symplectic geometry, which will be essential for the main results of this work. 

\begin{defi}
	A symplectic manifold is a pair $(M,\omega)$, where $M$ is a smooth manifold and $\omega$ is a closed $2$-form which yields a non-degenerate bilinear form at every point.
\end{defi}

Some of the most important examples of symplectic manifolds come from complex differential geometry. We recall that a manifold $M$ is complex if their charts map onto open sets of complex vector spaces so that the changes of coordinates are holomorphic. For such a manifold $M$, this yields a complex structure $J_z$ on every tangent space $T_z M$, for every $z \in M$. In turn, this provides a tensor field $J$ known as the complex structure tensor of $M$. In particular, we have $J^2 = -I$ the negative of the identity tensor acting on the fibers of the tangent bundle $TM$. We refer to \cite{Mok} for further details. 

The next definition describes well behaved Riemannian metrics with respect to these constructions.

\begin{defi}
	Let $M$ be a complex manifold with complex structure tensor $J$ and a given Riemannian metric $g$. We say that $M$ is a Hermitian manifold if it satisfies
	\[
		g_z(J_zu,J_zv)=g_z(u,v)
	\]
	for every $z \in M$ and $u,v \in T_z M$.
\end{defi}

We now proceed to relate Hermitian manifolds to symplectic geometry. We will explain the main constructions and refer to \cite{Mok} for further details. Let us start by considering a complex manifold $M$ with complex structure tensor $J$. Then, the tangent bundle can be complexified to a complex tangent bundle denoted by $T^\mathbb{C} M$, and the action of $J$ on $TM$ can also be complexified to obtain a tensor $J^\mathbb{C}$ acting on $T^\mathbb{C} M$. Such complexifications are performed fiberwise.

Since $(J_z^{\mathbb{C}})^2=-I$, for every $z \in M$, if we define the spaces
\begin{align*}
	T_z^{1,0}M &= \{v\in T_z^{\mathbb{C}}M 
		\mid J_z^{\mathbb{C}}v=iv\}, \\
	T_z^{0,1}M &= \{v\in T_z^{\mathbb{C}}M 
		\mid J_z^{\mathbb{C}}v=-iv\},
\end{align*}
then we have $T_z^{\mathbb{C}}M = T_z^{1,0}M \oplus T_z^{0,1}M$. These spaces are known as the subspaces of holomorphic and anti-holomorphic tangent vectors. If $(z_1, \dots, z_n)$ is a holomorphic chart with real components obtained from the decomposition $z_j=x_j+iy_j$, then the usual Wirtinger differential operators are given by
\[
	\frac{\partial}{\partial z_j}
		= \frac{1}{2} \left(
			\frac{\partial}{\partial x_j} 
			- i\frac{\partial}{\partial y_j}\right), \quad
	\frac{\partial}{\partial \overline{z}_j} 
		= \frac{1}{2}\left(\frac{\partial}{\partial x_j}
			+ i\frac{\partial}{\partial y_j}\right),
\]
for every $j = 1, \dots, n$. The first set of operators define at every point in the domain of the chart a basis for the corresponding fibers of $T^{1,0} M$. Similarly, the second set of operators define a basis for the fibers of $T^{0,1} M$. The corresponding dual basis are given by
\[
	\dif z_j = \dif x_j + i \dif y_j, \quad
	\dif \overline{z}_j = \dif x_j - i \dif y_j,
\]
where $j = 1, \dots, n$.

Let us now consider a Riemannian metric $g$ on $M$ for which $M$ is a Hermitian manifold. We can complexify $g$ to a complex bilinear tensor $g^\mathbb{C}$ defined on the complexified tangent bundle $T^\mathbb{C} M$. This yields a positive definite Hermitian form
\begin{align*}
		T_z^{1,0}M \times T_z^{1,0}M 
			&\longrightarrow\mathbb{C}, \\
		(u,v) &\longmapsto g_z^\mathbb{C}(u,\overline{v}),
\end{align*}
for every $z \in M$. In local coordinates, this can be written as
\[
	\sum_{j,k=1}^n g_{jk}(z)dz_j\otimes d\overline{z}_k.
\]
For this reason, we will denote this field of complex Hermitian forms with the same symbol $g$. To more easily distinguish between the two of them, we will refer to the original $g$ as the Riemannian metric of $M$ and we will call the previous field of Hermitian forms the Hermitian metric of $M$.

The previous setup and constructions allow to introduce the next important geometric object.

\begin{defi}\label{defi:Kahler}
	For a Hermitian manifold $M$ with Hermitian metric $g$ as constructed above, the associated $2$-form is given by
	\[
		\omega = g(J(\cdot), \cdot) = -2 \mathrm{Im}(g)
	\]
	where the first occurrence of $g$ is the Riemannian metric and the second one is the corresponding Hermitian metric. The Hermitian manifold $M$ is called K\"ahler if its associated $2$-form is closed. In this case, $\omega$ is called the K\"ahler form of $M$.
\end{defi}

It is straightforward to check that the associated $2$-form of any Hermitian manifold is non-degenerate. Hence, every K\"ahler manifold is a symplectic manifold, and in this case the K\"ahler form is its symplectic form.

One can alternatively provide a K\"ahler structure on a complex manifold by introducing a field of Hermitian bilinear forms. This is the content of the next result which is a particular case of Proposition~1 in page~18 from \cite{Mok}.

\begin{propo}\label{propo:KahlerFromHermitian}
	Let $M$ be a complex manifold and let $g$ be a tensor field of positive definite Hermitian bilinear forms on $T^{1,0} M$. Assume that for every holomorphic coordinate chart $(z_1, \dots, z_n)$, in a family of charts covering $M$, there is some real valued function $F$ such that
	\[
		g = \sum_{j,k=1}^n 
			\frac{\partial^2 F}{\partial z_j \partial \overline{z}_k}
				\dif z_j \otimes \dif \overline{z}_k
	\]
	in the domain of the given chart. Then, the tensor $2\mathrm{Re}(g)$ is a Riemannian metric that yields a K\"ahler structure on $M$ whose Hermitian metric is given by $g$.
	
\end{propo}

\subsection{The Bergman metric and its Kähler form}
\label{subsec:BergmanmetricKahler}
We now use the results previously obtained to construct a K\"ahler structure on the Cartan domains of type III. The next fundamental theorem is a particular case of the discussion in the first part of Chapter~4 in \cite{Mok} (see also~\cite{Helgason,Range}). Note that, from now on, we will use the canonical complex linear coordinates of $\mathrm{Symm}(n,\mathbb{C})$.

\begin{teo}\label{teo:BergmanMetric}
	Let $D$ be either of $D^{III}_n$ or $\mathscr{S}_n$ and let $K_D(Z,W)$ be the reproducing Bergman kernel of $D$. Then, the tensor given~by
	\[
		\sum_{\substack{1 \leq l \leq m \leq n \\ 
				1 \leq j \leq k \leq n}}
			\frac{\partial^2\log K_D(Z,Z)}%
			{\partial z_{lm}\partial \overline{z}_{jk}}
			\dif z_{lm} \otimes \dif\overline{z}_{jk},
	\]
	is a field of positive definite Hermitian forms that yields a structure of K\"ahler manifold on $D$. Furthermore, both the corresponding Riemannian metric and associated K\"ahler form are invariant under the group of biholomorphisms.
\end{teo}

We use Theorem~\ref{teo:BergmanMetric} to introduce K\"ahler structures on $D^{III}_n$ and $\mathscr{S}_n$ by normalizing the tensor considered in its statement. These normalization will simplify some formulas below.

\begin{defi}
	Let $D$ be either of $D^{III}_n$ or $\mathscr{S}_n$ and $K_D(Z,W)$ the Bergman kernel of $D$. The Bergman metric of $D$ is the field of Hermitian forms given~by
	\[
		g_D = 
		c_D \sum_{\substack{1 \leq l \leq m \leq n \\ 
				1 \leq j\leq k \leq n}}
			\frac{\partial^2\log K_D(Z,Z)}%
				{\partial z_{lm}\partial \overline{z}_{jk}}
					dz_{lm}\otimes d\overline{z}_{jk},
	\]
	where $c_{D^{III}_n} = \frac{1}{n+1}$ and $c_{\mathscr{S}_n} = \frac{4}{n+1}$.
\end{defi}

The next two results are very well known properties of the Wirtinger differential operators that will be useful in this work. We state them for the sake of completeness.

\begin{lema}
	For any smooth function $f : \mathbb{C}^N \longrightarrow \mathbb{C}$ we have
	\begin{align*}
		\dif f = 
		\sum_{j=1}^N \bigg(
		\frac{\partial f}{\partial 	z_j} \dif z_j
		+ \frac{\partial f}{\partial \overline{z}_j} 
		\dif \overline{z}_j\bigg).
	\end{align*}
\end{lema}

\begin{lema}[Chain rule for Wirtinger derivatives]
\label{lem:ChainRuleWirtinger}
	Let $g : \mathbb{C}^n \rightarrow \mathbb{C}^m$ and $f : \mathbb{C}^m \rightarrow \mathbb{C}$ be smooth functions. Then, we have
	\begin{align*}
		\frac{\partial(f\circ g)}{\partial z_j}
		&= \sum_{k=1}^m \bigg(\frac{\partial f}{\partial z_k}\circ g 
		\frac{\partial g_k}{\partial z_j}
		+\frac{\partial f}{\partial \overline{z}_k}\circ g \frac{\partial \overline{g}_k}{\partial z_j}\bigg), \\
		\frac{\partial(f\circ g)}{\partial \overline{z}_j}
		&= \sum_{k=1}^{m} \bigg(\frac{\partial f}{\partial z_k}\circ g
		\frac{\partial g_k}{\partial \overline{z}_j}
		+\frac{\partial f}{\partial \overline{z}_k}\circ g
		\frac{\partial \overline{g}_k}{\partial \overline{z}_j}\bigg).
	\end{align*}
\end{lema}

The following elementary computation will be used latter on. We provide its proof for the sake of completeness.

\begin{lema}\label{lem:derdet}
	The differential of $\det : \mathrm{Mat}(n,\mathbb{C}) \rightarrow \mathbb{C}$ is given~by
	\begin{align*}
		\dif\;(\det)_A = 
		\mathrm{tr}\left(\mathrm{adj}(A)\mathrm{d}A\right), 
	\end{align*}
	for every $A \in \mathrm{Mat}(n,\mathbb{C})$, where $\mathrm{adj}(A)$ (adjugate of $A$) is the transpose of the cofactor matrix of $A$.
\end{lema}
\begin{proof}
	If $A=(a_{lm}) \in \mathrm{Mat}(n,\mathbb{C})$ and $c_{lm}$ is the cofactor of $a_{lm}$, then the cofactor expansion of the determinant along the $k$-th column is given by
	\begin{align*}
	\det A = 
		\sum_{l=1}^{n} c_{lk}a_{lk} =
		 \sum_{l=1}^{n}\Big(\mathrm{adj}(A)^T\Big)_{lk} a_{lk}.
	\end{align*}
	It follows that 
	\[
	\frac{\partial\det}{\partial a_{jk}}(A)
		= c_{jk} = \Big(\mathrm{adj}(A)^T\Big)_{jk}
	\]
	and we obtain the differential
	\begin{align*}
		\dif\;(\det)_A 
		&= \sum_{j,k=1}^n 
			\frac{\partial\det}{\partial a_{jk}}(A) \dif a_{jk}
			= \sum_{j,k=1}^n c_{jk} \dif a_{jk} \\
		&= \sum_{j,k=1}^n
			\Big(\mathrm{adj}(A)^T\Big)_{jk} \dif a_{jk}
			=\sum_{j,k=1}^n \big(\mathrm{adj}(A) \big)_{kj} 
					\dif a_{jk} \\
		&= \sum_{k=1}^n (\mathrm{adj}(A) \dif A)_{kk}
			=\mathrm{tr}(\mathrm{adj}(A) \dif A).
\end{align*}
\end{proof}

We now obtain explicit formulas for the Bergman metrics of the Cartan domains of type III. Note that we have provided coordinate free expressions. This will be useful for our computations in the rest of this work.

\begin{teo}\label{teo:BergmanMetricFormula}
	The Bergman metrics on $D^{III}_n$ and $\mathscr{S}_n$ are respectively given~by
	\begin{align*}
		g_Z^{D^{III}_n}(U,V)
			&=\mathrm{tr} \big((I_n-Z\overline{Z})^{-1} U 
			(I_n-\overline{Z}Z)^{-1}\overline{V}\big), \\
		g_Z^{\mathscr{S}_n}(U,V)
			&=\mathrm{tr} \big(\mathrm{Im}(Z)^{-1} U 
			\mathrm{Im}(Z)^{-1} \overline{V}\big),
	\end{align*}
	for every $U, V \in \mathrm{Symm}(n,\mathbb{C})$. In particular, the K\"ahler forms of $D^{III}_n$ and $\mathscr{S}_n$ are respectively given~by
	\begin{align*}
		\omega_Z^{D^{III}_n}(U,V)
			=&\;i\;\mathrm{tr}
			\big((I_n-Z\overline{Z})^{-1} U 
			(I_n-\overline{Z}Z)^{-1}\overline{V} \big)  \\
			 	&-i\;\mathrm{tr}\big((I_n-\overline{Z}Z)^{-1}
			 	\overline{U} (I_n-Z\overline{Z})^{-1}V\big),\\
		\omega_Z^{\mathscr{S}_n}(U,V)
			=&\;2\;\mathrm{tr}\big(\mathrm{Im}(Z)^{-1}\mathrm{Re}(U)
			\mathrm{Im}(Z)^{-1}\mathrm{Im}(V) \big) \\
				&-2\;\mathrm{tr}\big(\mathrm{Im}(Z)^{-1} \mathrm{Im}(U)
				\mathrm{Im}(Z)^{-1}\mathrm{Re}(V)\big),
	\end{align*}
	for every $U, V \in \mathrm{Symm}(n, \mathbb{C})$.
\end{teo}
\begin{proof}
	In this proof we will consider the complex vector spaces $\mathrm{Symm}(n,\mathbb{C})$ and $\mathrm{Mat}(n,\mathbb{C})$ whose coordinates will be denoted in both cases by $z_{jk}$, even though they have different meanings for such spaces. However, from the context where these coordinates are used it will be easy to identify the actual meaning.

	We start by computing the Bergman metric on $D^{III}_n$. First, we observe that we have the following partial derivative
	\[
		\frac{\partial (I_n - Z \overline{Z})}%
			{\partial \overline{z}_{jk}}(Z) 
			= -Z \frac{\partial \overline{Z}}%
				{\partial \overline{z}_{jk}}=-ZE_{jk},
	\]
	where $E_{jk}$ is the $n\times n$ symmetric matrix that has $1$ in the entries $(j,k)$ and $(k,j)$ and $0$ elsewhere. Note that these matrices are the basis with respect to which we are considering the canonical coordinates in $\mathrm{Symm}(n,\mathbb{C})$.
	
	Next, using the previous computation, applying Lemmas~\ref{lem:derdet} and \ref{lem:ChainRuleWirtinger}, Equation~\eqref{eq:KDIIIn} and using the fact that $\det$ is holomorphic, we obtain
	\begin{align*}
	\frac{1}{n+1}\frac{\partial}{\partial \overline{z}_{jk}}
		& \log K_{D^{III}_n}(Z,Z) = \\
		&= \frac{1}{\det(I_n-Z\overline{Z})}
			\sum_{l,m=1}^n
				\frac{\partial\det}{\partial z_{lm}}
				(I_n-Z\overline{Z})(ZE_{jk})_{lm}  \\
		&= \frac{1}{\det(I_n-Z\overline{Z})}
			\sum_{l,m=1}^n
			(\mathrm{adj}(I_n-Z\overline{Z})^T)_{lm}
			(ZE_{jk})_{lm}  \\
		&= \frac{1}{\det(I_n-Z\overline{Z})}
			\mathrm{tr}\big(\mathrm{adj}(I_n-Z\overline{Z})
				ZE_{jk}\big)  \\
		&= \mathrm{tr}\big((I_n-Z\overline{Z})^{-1} ZE_{jk}\big).
	\end{align*}
	Now, we will use the easy to prove relations
	\[
		(I_n-Z\overline{Z})^{-1}Z 
			= Z(I_n-\overline{Z}Z)^{-1}, \quad
		\overline{Z}(I_n-Z\overline{Z})^{-1} 
			= (I_n-\overline{Z}Z)^{-1}\overline{Z},
	\]
	which hold for every $Z \in D^{III}_n$. Using the identities obtained so far we compute
	\begin{align*}
		\frac{1}{n+1} 
		\frac{\partial^2}{\partial z_{lm}\partial\overline{z}_{jk}}
			&\log K^{D^{III}_n}(Z,Z)
				= \frac{\partial}{\partial z_{lm}} 
				\mathrm{tr}
				\big((I_n-Z\overline{Z})^{-1} Z E_{jk}\big)  \\
			= &\;\mathrm{tr}
				\big((I_n-Z\overline{Z})^{-1} 
				E_{lm} \overline{Z} 
					(I_n-Z\overline{Z})^{-1}Z E_{jk}\big) \\
			&+ \mathrm{tr}\big((I_n-Z\overline{Z})^{-1} 
				E_{lm} E_{jk}\big)  \\
			= &\;\mathrm{tr}
				\big((I_n-Z\overline{Z})^{-1} E_{lm} 
				(I_n-\overline{Z}Z)^{-1}\overline{Z} 
				Z E_{jk}\big)  \\
			&+ \mathrm{tr}
				\big((I_n-Z\overline{Z})^{-1} 
				E_{lm} E_{jk} \big)  \\
			= &\;\mathrm{tr}
				\big((I_n-Z\overline{Z})^{-1} E_{lm} 
				(I_n-\overline{Z}Z)^{-1}
				(\overline{Z}Z+(I_n-\overline{Z}Z)) E_{jk} \big) \\
			= &\;\mathrm{tr}
				\big((I_n-Z\overline{Z})^{-1} E_{lm} 
				(I_n-\overline{Z}Z)^{-1} E_{jk} \big).
	\end{align*}
	This implies that the metric $g^{D^{III}_n}_Z$ satisfies the required identity on the basic elements of the vector space $\mathrm{Symm}(n,\mathbb{C})$, thus proving the result for the Bergman metric of $D^{III}_n$. The corresponding computation of the Bergman metric for $\mathscr{S}_n$ is obtained similarly.

	From Definition~\ref{defi:Kahler} the K\"ahler form of $D^{III}_n$ is given by
	\begin{align*}
		\omega^{D^{III}_n}_Z(U,V) &= 
			-2 \mathrm{Im}\Big(g^{D^{III}_n}_Z(U,V)\Big) \\
			&= i \Big(g^{D^{III}_n}_Z(U,V) - 
				\overline{g^{D^{III}_n}_Z(U,V)}\Big),
	\end{align*}
	which yields the stated formula from our computation of the Bergman metric of $D^{III}_n$.

	Finally, for the K\"ahler form of $\mathscr{S}_n$ we compute
	\begin{multline*}
		\omega_Z^{\mathscr{S}_n}(U,V) =
			-2\mathrm{Im}\Big(g_Z^{\mathscr{S}_n}(U,V)\Big) \\ 
			= -2\mathrm{Im}\Big( 
				\mathrm{tr}\big(
				\mathrm{Im}(Z)^{-1}
				(\mathrm{Re}(U)+i\mathrm{Im}(U))
				\mathrm{Im}(Z)^{-1}
				(\mathrm{Re}(V)-i\mathrm{Im}(V))
				\big)
			\Big),
	\end{multline*}
	from which the stated formula is easily obtained.
\end{proof}

\subsection{Moment maps}
Now we turn back our attention to symplectic geometry. It will provide the main geometric tools and objects that we will apply to the study of Toeplitz operators. We refer to \cite{McDuffSalamon} for the symplectic geometry facts stated without proof.

In the rest of this subsection $(M, \omega)$ will denote a fixed symplectic manifold. A diffeomorphism $\varphi : M \rightarrow M$ is called a symplectomorphism if $\varphi^*(\omega)=\omega$. In other words, a symplectomorphism is a diffeomorphism preserving the symplectic form.

If $H$ is a Lie group with a smooth action on $M$, then we say that the $H$-action is symplectic if the map
\begin{align*}
	M & \longrightarrow M \\
	z &\longmapsto h\cdot z
\end{align*}
is a symplectomorphism for every $h \in H$.

There are two important types of vector fields on $M$. From now on, we will denote by $\mathfrak{X}(M)$ the Lie algebra of vector fields over $M$. A field $X \in \mathfrak{X}(M)$ is called a symplectic vector field if and only if the $1$-form $\omega(X, \cdot)$ is closed, and it is called a Hamiltonian vector field if and only if the form $\omega(X, \cdot)$ is exact. We will denote by $\mathfrak{X}(M,\omega)$ the space of symplectic vector fields on $M$. It is a well known fact that $\mathfrak{X}(M,\omega)$ is a Lie subalgebra of $\mathfrak{X}(M)$.

For any smooth function $f : M \rightarrow \mathbb{R}$, the non-degeneracy of $\omega$ implies the existence of a unique element $X_f \in \mathfrak{X}(M)$ such that
\[
	\dif f = \omega(X_f, \cdot).
\]
In this case, $X_f$ is called the Hamiltonian vector field associated to $f$. 

Symplectic vector fields can be characterized by symplectomorphisms. More precisely, it is well known that an element $X \in \mathfrak{X}(M)$ belongs to $\mathfrak{X}(M,\omega)$ if and only if the local flow generated by $X$ acts by (locally defined) symplectomorphisms.

An important converse to the previous fact relates symplectic actions to symplectic vector fields as follows. Let us consider a symplectic action of a Lie group $H$ on $M$. Then, for every $X \in \mathfrak{h}$ (the Lie algebra of $H$), we define the induced vector field on $M$ by
\[
	X^\sharp_z = \frac{\dif}{\dif s}\sVert[2]_{s=0}\exp(sX)\cdot z. 
\]
for every $z \in M$, where $\exp : \mathfrak{h} \rightarrow H$ is the exponential map of $H$. Then, the fact that the $H$-action is symplectic implies that $X^\sharp \in \mathfrak{X}(M,\omega)$ for every $X \in \mathfrak{h}$.

In the previous discussion, we have shown two different constructions that map into the space $\mathfrak{X}(M,\omega)$ of symplectic vector fields. Hence, a natural problem to consider is the existence of a map $\mathfrak{h} \rightarrow C^\infty(M)$ that makes the following diagram commute
\[
	\xymatrix{
			& C^\infty(M) \ar[d] \\
		\mathfrak{h} \ar[ur] \ar[r] & \mathfrak{X}(M,\omega)
	}
\]
where the vertical arrow is the map $f \mapsto X_f$ and the horizontal arrow is the map $X \mapsto X^\sharp$. The existence of such diagonal map yields the notion of a moment map for the $H$-action. The precise definition requires some additional conditions. We recall that $\mathrm{Ad} = \mathrm{Ad}_H : H \rightarrow \mathrm{GL}(\mathfrak{h})$ denotes the adjoint representation of the Lie group $H$, and that $\mathrm{Ad}^*$ denotes the dual representation on $\mathfrak{h}^*$. In particular, we have $\mathrm{Ad}^*(h) = \mathrm{Ad}(h^{-1})^\top$ for every $h \in H$.

\begin{defi}\label{defi:momentmap}
	Let $(M,\omega)$ be a symplectic manifold and let $H$ be a Lie group acting by symplectomorphisms on $M$. A moment map for the $H$-action is a smooth map $\mu: M \rightarrow \mathfrak{h}^{*}$, where $\mathfrak{h}^*$ is the vector space dual of $\mathfrak{h}$, that satisfies the following properties.
	\begin{enumerate}
		\item For every  $X \in \mathfrak{h}$ consider the map $\mu_X : M \rightarrow \mathbb{R}$ given by $\mu_X(z) = \langle\mu(z), X\rangle$. Then, the Hamiltonian vector field associated to $\mu_X$ is $X^\sharp$, for every $X \in \mathfrak{h}$. In other words, it holds
		\[
			\dif\mu_X=\omega(X^{\#}, \cdot),
		\]
		for every $X \in \mathfrak{h}$.
		\item The map $\mu$ is $H$-equivariant. In other words, we have
		\[
			\mu(h\cdot z) = \mathrm{Ad}^*(h)(\mu(z)),
		\]
		for every $z\in M$ and $h\in H$.
	\end{enumerate}
\end{defi}

\begin{remark}
	If $H$ is an Abelian group, then its adjoint representation satisfies $\mathrm{Ad}(h) = I_\mathfrak{h}$ for every  $h\in H$. Hence, in this case condition~2.~in Definition~\ref{defi:momentmap} reduces to 
	\[
		\mu(h\cdot z) = \mu(z),
	\]
	for every $h \in H$ and $z \in M$. In other words, this requires the smooth map to be $H$-invariant.
\end{remark}

\section{Three Abelian biholomorphism groups and their moment maps}
\label{sec:3Groups}
In this section we study three special types of subgroups of biholomorphisms acting on Cartan domains of type III. For the corresponding Abelian groups, we compute the moment maps. We will see later on that these moment maps are a powerful tool to find commutative $C^*$-algebras generated by Toeplitz operators.

\subsection{Elliptic, Hyperbolic, and Parabolic Actions}
\label{subsec:ellhyperpar_actions}
The Cartan domains $D^{III}_n$ and their unbounded realizations $\mathscr{S}_n$ carry three interesting actions of subgroups of biholomorphisms. As we will see, these actions generalize the three different types of M\"obius transformations found for the unit disk $\mathbb{D}$ and the upper half plane $\mathbb{H}$.

Proposition~\ref{propo:biholoDIIIn} provides the action 
\begin{align*}
	\mathrm{U}(n) \times D^{III}_n &\longrightarrow D^{III}_n \\
	U \cdot Z &= UZU^\top,
\end{align*}
which yields the subgroup of biholomorphisms that fixes the origin. Up to conjugacy, this characterizes the subgroups that fix some point in the domain $D^{III}_n$. This is so because of the homogeneity of this domain. We will call this the \textbf{Elliptic Action} on $D^{III}_n$.

Next we observe that there is a canonical homomorphism of Lie groups given by
\begin{align*}
	\mathrm{GL}(n,\mathbb{R}) &\longrightarrow 
	\mathrm{Sp}(n,\mathbb{R}) \\
	A &\longmapsto 
	\begin{pmatrix}
		A & 0 \\
		0 & (A^{-1})^\top
	\end{pmatrix}.
\end{align*}
A straightforward computation shows that this assignment is indeed a homomorphism whose image lies in $\mathrm{Sp}(n,\mathbb{R})$. Hence,  this homomorphism and Proposition~\ref{propo:biholoSn} provide the action
\begin{align*}
	\mathrm{GL}(n,\mathbb{R}) \times \mathscr{S}_n &\longrightarrow
		\mathscr{S}_n \\
	A \cdot Z &= AZA^\top.
\end{align*}
It is easily seen that this action realizes the subgroup of biholomorphisms that fixes the origin, a boundary point of the domain $\mathscr{S}_n$. For this reason, we will call this the \textbf{Hyperbolic Action} on $\mathscr{S}_n$.

Finally, we have a canonical homomorphism of Lie groups given by 
\begin{align*}
	\mathrm{Symm}(n,\mathbb{R}) &\longrightarrow
		\mathrm{Sp}(n,\mathbb{R}) \\
	S &\longmapsto 
		\begin{pmatrix}
			I_n & S \\
			0 & I_n
		\end{pmatrix},
\end{align*}
where $\mathrm{Symm}(n,\mathbb{R})$ is endowed with the Lie group structure with operation given by the sum of matrices. Again, it is straightforward to show that this map is indeed a homomorphism into the group $\mathrm{Sp}(n,\mathbb{R})$. We now have that this homomorphism together with Proposition~\ref{propo:biholoSn} provide the action
\begin{align*}
	\mathrm{Symm}(n,\mathbb{R}) \times \mathscr{S}_n 
		&\longrightarrow \mathscr{S}_n \\
	S \cdot Z &= Z + S.
\end{align*}
This action realizes the subgroup of biholomorphisms of the tube type domain $\mathscr{S}_n$ that correspond to translations on the real vector space part. Since this action clearly generalizes the translation action on the real part on the upper half-plane $\mathbb{H}$, we will call this action on $\mathscr{S}_n$ the \textbf{Parabolic Action}.

In fact, all three actions introduced above generalize the behavior observed in the $1$-dimensional case. This is stated in the following well known result. We recall that two biholomorphisms are conjugated if they are so under some other biholomorphism. This result justifies our choice of notation for the actions considered above.

\begin{coro}\label{coro:3Actionn=1}
	Let us denote by $D$ either $\mathbb{D}$ or $\mathbb{H}$. If $\varphi$ is a biholomorphism of $D$, then the following equivalences hold.
	\begin{enumerate}
		\item The M\"obius transformation $\varphi$ is elliptic if and only if it is conjugated to a transformation that belongs to the action $\mathbb{T} \times \mathbb{D} \rightarrow \mathbb{D}$ given by $z \mapsto tz$.
		\item The M\"obius transformation $\varphi$ is hyperbolic if and only if it is conjugated to a transformation that belongs to the action $\mathbb{R}_+ \times \mathbb{H} \rightarrow \mathbb{H}$ given by $z \mapsto rz$.
		\item The M\"obius transformation $\varphi$ is parabolic if and only if it is conjugated to a transformation that belongs to the action $\mathbb{R} \times \mathbb{H} \rightarrow \mathbb{H}$ given by $z \mapsto z + s$.
	\end{enumerate}
\end{coro}

We note that the Elliptic and Hyperbolic Actions are given by actions of Abelian groups if and only if $n = 1$. Nevertheless, the Parabolic Action is given by an Abelian Lie group in any dimension. For these reason, we introduce in the next definition actions of Abelian groups associated to the Elliptic and Hyperbolic cases. 

\begin{defi}\label{defi:AbelianEllHyp}
	The Abelian Elliptic Action on $D^{III}_n$ is defined by
	\begin{align*}
		\mathbb{T} \times D^{III}_n &\longrightarrow D^{III}_n \\
			t \cdot Z &= t^2 Z.
	\end{align*}
	The Abelian Hyperbolic Action on $\mathscr{S}_n$ is defined by 
	\begin{align*}
		\mathbb{R}_+ \times \mathscr{S}_n 
			&\longrightarrow \mathscr{S}_n \\
			r \cdot Z &= r^2 Z.
	\end{align*}
\end{defi}

\begin{remark}\label{rmk:AbelianActions}
	We note that the Abelian Elliptic and Abelian Hyperbolic actions are obtained by considering the center of the groups corresponding to the non-Abelian actions. More precisely, we have the centers
	\[
		\mathrm{Z}(\mathrm{U}(n)) = \mathbb{T} I_n, \quad
		\mathrm{Z}(\mathrm{GL}(n,\mathbb{R})) 
			= \mathbb{R}_+ I_n \cup (-\mathbb{R}_+ I_n),
	\]
	and the actions in Definition~\ref{defi:AbelianEllHyp} are the restriction of the previously defined actions to the identity connected components of these center groups.	On the other hand, $\mathrm{Symm}(n,\mathbb{R})$ is already Abelian so that it coincides with its center, in other words we have
	\[
		\mathrm{Z}(\mathrm{Symm}(n,\mathbb{R})) =
			\mathrm{Symm}(n,\mathbb{R}).
	\]
	Hence, the most obvious definition of ``Abelian Parabolic Action'' would yield what we already have defined as the Parabolic Action. We also observe that these three actions of Abelian groups of biholomorphisms, the Abelian Elliptic, Abelian Hyperbolic and Parabolic, are natural generalizations of the actions described in Corollary~\ref{coro:3Actionn=1}.
\end{remark}

\subsection{Moment maps of the Abelian actions}
\label{subsec:momentmaps_Abelian}
We will now compute moment maps for all three Abelian actions introduced in this section. We refer to Definition~\ref{defi:AbelianEllHyp} and Remark~\ref{rmk:AbelianActions}. It follows from Theorem~\ref{teo:BergmanMetric} that every biholomorphism of either of the domains $D^{III}_n$ and $\mathscr{S}_n$ preserves the corresponding K\"ahler form. Hence, all the groups considered above act by symplectomorphisms. In particular, the notion of moment map given in Definition~\ref{defi:momentmap} can be applied to such actions.

\subsubsection{Moment map of the Abelian Elliptic Action}
The group in this case is $\mathbb{T}$ acting on $D^{III}_n$. The Lie algebra of this group is $\mathbb{R}$. The latter is canonically isomorphic to its dual $\mathbb{R}^*$, so we will compute a moment map as a function $D^{III}_n \rightarrow \mathbb{R}$.

For every element $t \in \mathbb{R}$ the corresponding induced vector field on $D^{III}_n$ is given~by
\[
	t_Z^\sharp 
		= \frac{\dif}{\dif s}\sVert[2]_{s=0} \exp(st) \cdot Z
		= \frac{\dif}{\dif s}\sVert[2]_{s=0} \exp(2ist)Z = 2itZ,
\]
for every $Z \in D^{III}_n$. Note that we have used the fact that the (Lie group) exponential map $\mathbb{R} \rightarrow \mathbb{T}$ satisfies $t \mapsto \exp(it)$.

\begin{propo}\label{propo:MomentMapT}
	The function given by
	\begin{align*}
		\mu^\mathbb{T} : D^{III}_n &\longrightarrow
			\mathbb{R} \\
		\mu^\mathbb{T}(Z) 
			&= -2 \mathrm{tr} \big(
			(I_n - Z \overline{Z})^{-1}
			\big),
	\end{align*}
	is a moment map for the Abelian Elliptic Action on $D^{III}_n$.
\end{propo}
\begin{proof}
	We start by computing $\omega_Z^{D^{III}_n}(t^\sharp_Z,\cdot)$ for every $t \in \mathbb{R}$ and $Z \in D^{III}_n$. For this first computation we use the above formula for $t^\sharp$ and the expression for the K\"ahler form of $D^{III}_n$ obtained in Theorem~\ref{teo:BergmanMetricFormula}. We have
	\begin{align*}
		\omega_Z^{D^{III}_n}(t^\sharp_Z,V)
			=&\;i\;\mathrm{tr}
				\big((I_n-Z\overline{Z})^{-1} 2itZ 
				(I_n-\overline{Z}Z)^{-1}\overline{V} \big)  \\
			&-i\;\mathrm{tr}\big((I_n-\overline{Z}Z)^{-1}
				\overline{2itZ} (I_n-Z\overline{Z})^{-1}V\big),\\
			=&\;-2t\;\mathrm{tr}
				\big((I_n-Z\overline{Z})^{-1} Z 
					(I_n-\overline{Z}Z)^{-1}\overline{V} \big)  \\
			&-2t\;\mathrm{tr}\big((I_n-\overline{Z}Z)^{-1}
				\overline{Z} (I_n-Z\overline{Z})^{-1}V\big),\\
	\end{align*}
	for every $V \in \mathrm{Symm}(n,\mathbb{C})$.

	On the other hand, we consider the function $\mu_t : D^{III}_n \rightarrow \mathbb{R}$ defined by
	\[
		\mu_t(Z) = \langle\mu^\mathbb{T}(Z),t \rangle =
					t \mu^\mathbb{T}(Z),
	\]
	and we compute its differential as follows
	\begin{align*}
		\dif\;(\mu_t)_Z&(V) = \\
			=&\; -2t\frac{\dif}{\dif s}\sVert[2]_{s=0}
				\mathrm{tr}\Big(\big(I_n - 
					(Z+sV)\overline{(Z+sV)}\big)^{-1}\Big) \\
			=&\; -2t\;\mathrm{tr}\Big(
				\big(I_n - Z\overline{Z}\big)^{-1}
					\big(V\overline{Z} + Z \overline{V}\big)
					\big(I_n - Z\overline{Z}\big)^{-1}		
					\Big) \\
			=&\; -2t\;\mathrm{tr}\Big(
					\big(I_n - Z\overline{Z}\big)^{-1} V
					\big(I_n - \overline{Z}Z\big)^{-1}
						\overline{Z}	
					\Big)  \\
				& -2t\;\mathrm{tr}\Big(
					Z\big(I_n - \overline{Z}Z\big)^{-1} 
						\overline{V}
					\big(I_n - Z\overline{Z}\big)^{-1}
					\Big)
	\end{align*}
	where we applied in the last identity the commutation relations between $Z$, $(I_n - Z\overline{Z})^{-1}$ and their conjugates used in the proof of Theorem~\ref{teo:BergmanMetricFormula}. We conclude that
	\[
		\dif\; (\mu_t)_Z(V) = \omega_Z^{D^{III}_n}(t^\sharp_Z,V)
	\]
	for every $V \in \mathrm{Symm}(n,\mathbb{C})$ and $Z \in D^{III}_n$. It follows that the first condition in Definition~\ref{defi:momentmap} is satisfied by the map in the statement. It remains to prove the $\mathbb{T}$-invariance of this map, but this is established through the identities
	\begin{align*}
		\mu^\mathbb{T}(t\cdot Z) = \mu^\mathbb{T}(t^2 Z) 
			&= -2\; \mathrm{tr}\big((I_n - t^2 Z \overline{t^2 Z})^{-1}\big) \\
			&= -2\; \mathrm{tr}\big((I_n - Z \overline{Z})^{-1}\big) 
				= \mu^\mathbb{T}(Z)
	\end{align*}
	that hold for every $t \in \mathbb{T}$ and $Z \in D^{III}_n$.
\end{proof}

\begin{remark}\label{rmk:AbelianElliptic_n=1}
	For the case $n=1$, the Abelian Elliptic Action yields the $\mathbb{T}$-action on  the unit disk $\mathbb{D}$ given by $t\cdot z = t^2z$. With this assumption, Proposition~\ref{propo:MomentMapT} provides the moment map
	\[
		\mu^\mathbb{T}(z) = -2\frac{1}{1 - |z|^2}.
	\]
	We observe that for actions of Abelian groups we can add to a given moment map an arbitrary, but fixed, constant to obtain another moment map (see Definition~\ref{defi:momentmap}). Hence, the map given by
	\[
		\mu(z) = \mu^\mathbb{T}(z) + 2 = -2 \frac{|z|^2}{1 - |z|^2},
	\]
	is a moment map as well for our $\mathbb{T}$-action on $\mathbb{D}$. This recovers, up to the multiplicative constant $2$, the moment map obtained in \cite[Proposition~4.1]{QSJFAUnitBall} for $n=1$. This referenced result computes the moment map for the natural action of the $n$-dimensional torus on the $n$-dimensional unit ball. We note that the factor $2$ comes from the reparameterization involved in using the action $t\cdot z = t^2 z$ instead of the action $t\cdot z = tz$.
\end{remark}

\subsubsection{Moment map of the Abelian Hyperbolic Action.}
We now have the group $\mathbb{R}_+$ acting on $\mathscr{S}_n$. The Lie algebra of this group is $\mathbb{R}$, which is canonically isomorphic to its dual. Hence, the moment map will be computed as a function $\mathscr{S}_n \rightarrow \mathbb{R}$.

For every $t \in \mathbb{R}$ the induced vector field on $\mathscr{S}_n$ is obtained as follows. This computation uses the fact that the (Lie group) exponential map is given in this case by $t \mapsto \exp(t)$.
\[
	t^\sharp_Z = \frac{\dif}{\dif s}\sVert[2]_{s=0} \exp(st)\cdot Z
		= \frac{\dif}{\dif s}\sVert[2]_{s=0} \exp(2st)Z = 2tZ,
\]
for every $Z \in \mathscr{S}_n$.

\begin{propo}\label{propo:MomentMapR+}
	The function given by
	\begin{align*}
		\mu^{\mathbb{R}_+} : \mathscr{S}_n &\longrightarrow \mathbb{R} \\
		\mu^{\mathbb{R}_+}(Z) &= 
			-4 \mathrm{tr}\big(
				\mathrm{Im}(Z)^{-1} \mathrm{Re}(Z)
			\big)
	\end{align*}
	is a moment map for the Abelian Hyperbolic Action on $\mathscr{S}_n$.	
\end{propo}
\begin{proof}
	We compute $\omega_Z^{\mathscr{S}_n}(t^\sharp_Z, \cdot)$, for every $t \in \mathbb{R}$ and $Z \in \mathscr{S}_n$. For this we use the previous computations and the expression of the K\"ahler form of $\mathscr{S}_n$ obtained in Theorem~\ref{teo:BergmanMetricFormula}. We have in this case
	\begin{align*}
		\omega_Z^{\mathscr{S}_n}(t^\sharp_Z, V) 
			=&\;2\;\mathrm{tr}\big(
					\mathrm{Im}(Z)^{-1} \mathrm{Re}(2tZ) \mathrm{Im}(Z)^{-1}
							\mathrm{Im}(V)
				\big) \\
			&-2\;\mathrm{tr}\big(
					\mathrm{Im}(Z)^{-1} \mathrm{Im}(2tZ) \mathrm{Im}(Z)^{-1}
						\mathrm{Re}(V)
				\big) \\
			=&\;4t\;\mathrm{tr}\big(
					\mathrm{Im}(Z)^{-1} \mathrm{Re}(Z) \mathrm{Im}(Z)^{-1}
						\mathrm{Im}(V)
					\big) \\
			&-4t\;\mathrm{tr}\big(
					\mathrm{Im}(Z)^{-1}	\mathrm{Re}(V)
				\big),
	\end{align*}
	for every $V \in \mathrm{Symm}(n,\mathbb{C})$.
	
	On the other hand, we consider the function $\mu_t : \mathscr{S}_n \rightarrow \mathbb{R}$ given by
	\[
		\mu_t(Z) = \langle \mu^{\mathbb{R}_+}(Z), t\rangle = t\mu^{\mathbb{R}_+}(Z),
	\]
	for which we compute the differential as follows
	\begin{align*}
		\dif\;(\mu_t)_Z&(V) = \\
		&=-4t\frac{\dif}{\dif s}\sVert[2]_{s=0}
			\mathrm{tr}\big(
				\mathrm{Im}(Z + s V)^{-1} \mathrm{Re}(Z + s V)
				\big) \\
		&=-4t\frac{\dif}{\dif s}\sVert[2]_{s=0}
			\mathrm{tr}\big(
				\big(\mathrm{Im}(Z) + s \mathrm{Im}(V)\big)^{-1} 
					\big(\mathrm{Re}(Z) + s \mathrm{Re}(V)\big)
				\big) \\
		&=4t\;\mathrm{tr}\big(
				\mathrm{Im}(Z)^{-1} \mathrm{Im}(V) \mathrm{Im}(Z)^{-1}
					\mathrm{Re}(Z)
				\big) 
			-4t\;\mathrm{tr}\big(
				\mathrm{Im}(Z)^{-1} \mathrm{Re}(V)
				\big),
	\end{align*}
	for every $V \in \mathrm{Symm}(n, \mathbb{C})$. From this we conclude that
	\[
		\dif\;(\mu_t)_Z(V) = \omega_Z^{\mathscr{S}_n}(t^\sharp_Z, V),
	\]
	for every $V \in \mathrm{Symm}(n, \mathbb{C})$ and $Z \in \mathscr{S}_n$. Hence, by Definition~\ref{defi:momentmap} it remains to show that $\mu^{\mathbb{R}_+}$ is $\mathbb{R}_+$-invariant, and this is verified in the next computation
	\begin{align*}
		\mu^{\mathbb{R}_+}(r\cdot Z) = \mu^{\mathbb{R}_+}(r^2 Z) 
			&= -4\mathrm{tr}\big(
					\mathrm{Im}(r^2Z)^{-1} \mathrm{Re}(r^2Z)
				\big) \\
			&= -4\mathrm{tr}\big(
					\mathrm{Im}(Z)^{-1} \mathrm{Re}(Z)
				\big)
				= \mu^{\mathbb{R}_+}(Z),
	\end{align*}
	which holds for every $r \in \mathbb{R}_+$ and $Z \in \mathscr{S}_n$.
\end{proof}

\begin{remark}\label{rmk:AbelianHyperbolic_n=1}
	For $n = 1$, the Abelian Hyperbolic Action yields the $\mathbb{R}_+$-action on the upper half-plane $\mathbb{H}$ given by $r\cdot z = r^2z$. Under this restriction, from Proposition~\ref{propo:MomentMapR+} we obtain the moment map
	\[
		\mu^{\mathbb{R}_+}(z) = -4 \frac{\mathrm{Re}(z)}{\mathrm{Im}(z)}.
	\]
	This recover, up to a constant factor, the moment map obtained in \cite[Proposition~4.3]{QSJFAUnitBall} for $n=1$. In this case the factor comes from two sources. Firstly, we use the action $r\cdot z = r^2z$, instead of the action $r\cdot z = rz$ used in~\cite{QSJFAUnitBall}. Secondly, our formula for the K\"ahler form for $\mathscr{S}_1 = \mathbb{H}$ differs by a constant factor from the corresponding formula found in~\cite{QSJFAUnitBall}.
\end{remark}

\subsubsection{Moment map of the Parabolic Action}
Finally, we consider the group $\mathrm{Symm}(n,\mathbb{R})$ acting on $\mathscr{S}_n$. Since $\mathrm{Symm}(n,\mathbb{R})$ is a vector group, it follows that it coincides with its Lie algebra and its exponential map is the identity. There is a canonical isomorphism between $\mathrm{Symm}(n,\mathbb{R})$ and its dual space given by the positive definite inner product
\[
	\langle A, B\rangle=\mathrm{tr}(AB),
\]
defined for $A,B \in \mathrm{Symm}(n,\mathbb{R})$.

For every $S \in \mathrm{Symm}(n,\mathbb{R})$ the corresponding induced vector field on $\mathscr{S}_n$ satisfies for every $Z \in \mathrm{Symm}(n,\mathbb{C})$
\[
	S^\sharp_Z = \frac{\dif}{\dif s}\sVert[2]_{s=0} 
			\exp(sS) \cdot Z 
		= \frac{\dif}{\dif s}\sVert[2]_{s=0} (Z + sS) = S,
\]
which is the constant vector field with value $S$.

\begin{propo}\label{propo:MomentMapSymm}
	The function given by
	\begin{align*}
		\mu^{\mathrm{Symm}(n,\mathbb{R})} : 
			\mathscr{S}_n &\longrightarrow 
				\mathrm{Symm}(n,\mathbb{R}) \\
		\mu^{\mathrm{Symm}(n,\mathbb{R})}(Z) 
			&= -2\mathrm{Im}(Z)^{-1},
	\end{align*}
	is a moment map for the Parabolic Action on $\mathscr{S}_n$.
\end{propo}
\begin{proof}
	For every $S \in \mathrm{Symm}(n,\mathbb{R})$ and $Z \in \mathscr{S}_n$, using the above computations and Theorem~\ref{teo:BergmanMetricFormula} we obtain
	\begin{align*}
		\omega_Z^{\mathscr{S}_n}(S^\sharp_Z,V) 
			=&\;2\; \mathrm{tr}\big(
					\mathrm{Im}(Z)^{-1} \mathrm{Re}(S) \mathrm{Im}(Z)^{-1} \mathrm{Im}(V)
				\big) \\
			&- 2\; \mathrm{tr}\big(
					\mathrm{Im}(Z)^{-1} \mathrm{Im}(S) \mathrm{Im}(Z)^{-1} \mathrm{Re}(V)
				\big) \\
			=&\;2\; \mathrm{tr}\big(
					\mathrm{Im}(Z)^{-1} S \;
					\mathrm{Im}(Z)^{-1} \mathrm{Im}(V)
				\big),
	\end{align*}
	for every $Z \in \mathscr{S}_n$.
	
	On the other hand, we consider for every $S \in \mathrm{Symm}(n,\mathbb{R})$ the map $\mu_S : \mathscr{S}_n \rightarrow \mathrm{Symm}(n,\mathbb{R})$ defined by
	\[
		\mu_S(Z) = -2\mathrm{tr}\big(
						\mathrm{Im}(Z)^{-1} S
					\big),
	\]
	for which we compute
	\begin{align*}
		\dif\;(\mu_S)_Z(V) 
			&= -2\frac{\dif}{\dif s}\sVert[2]_{s=0} 
					\mathrm{tr}\big(\mathrm{Im}(Z + sV)^{-1} S
				\big) \\
			&= 2\;\mathrm{tr}\big(
					\mathrm{Im}(Z)^{-1} \mathrm{Im}(V)
					\mathrm{Im}(Z)^{-1} S
				\big),
	\end{align*}
	for every $V \in \mathrm{Symm}(n,\mathbb{C})$ and $Z \in \mathscr{S}_n$. This immediately yields
	\[
		\dif\;(\mu_S)_Z(V) = \omega_Z^{\mathscr{S}_n}(S^\sharp_Z,V),
	\]
	for every $V \in \mathrm{Symm}(n,\mathbb{C})$ and $Z \in \mathscr{S}_n$. By Definition~\ref{defi:momentmap} it remains to establish the $\mathrm{Symm}(n,\mathbb{R})$-invariance of $\mu^{\mathrm{Symm}(n,\mathbb{R})}$, and this achieved by noting that
	\begin{align*}
		\mu^{\mathrm{Symm}(n,\mathbb{R})}
			(S\cdot Z) 
			&= \mu^{\mathrm{Symm}(n,\mathbb{R})}(Z + S)
				= -2 \big(\mathrm{Im}(Z + S)^{-1} \big) \\
			&= -2 \big(\mathrm{Im}(Z)^{-1} \big)
				= \mu^{\mathrm{Symm}(n,\mathbb{R})}
					(Z) 
	\end{align*}
	for every $Z \in \mathscr{S}_n$ and $S \in \mathrm{Symm}(n,\mathbb{R})$.
\end{proof}

\begin{remark}\label{rmk:Parabolic_n=1}
	For $n=1$, the Parabolic Action yields the $\mathbb{R}$-action on the upper half-plane $\mathbb{H}$ given by $s\cdot z = z + s$. And in this situation, Proposition~\ref{propo:MomentMapSymm} provides the moment map
	\[
		\mu^\mathbb{R}(z) = -2 \frac{1}{\mathrm{Im}(z)}.
	\]
	As in the previous cases, this recovers, up to a constant factor, the moment map obtained in \cite[Proposition~4.2]{QSJFAUnitBall} for $n=1$. As in the case of Remark~\ref{rmk:AbelianHyperbolic_n=1} the factor comes from a different normalization of the K\"ahler form on this work and~\cite{QSJFAUnitBall}.
\end{remark}

\section{Toeplitz operators with special symbols}
\label{sec:ToeplitzSpecialSymbols}
We will now describe Toeplitz operators with special symbols using two related alternatives: symbols invariant under biholomorphism groups and symbols depending on the moment maps of such groups. Both cases yield, under suitable conditions, commutative $C^*$-algebras generated by Toeplitz operators. 

First we introduce a general notation. As before, in the rest of this work $D$ denotes either of the domains $D^{III}_n$ or $\mathscr{S}_n$. For $\mathcal{A} \subset L^\infty (D)$ a set of essentially bounded symbols, we denote by $\mathcal{T}^{(\lambda)}(\mathcal{A})$ the $C^*$-algebra generated by the Toeplitz operators $T_a^{(\lambda)}$ where $a \in \mathcal{A}$.

\subsection{Invariant symbols}\label{subsec:inv-symbols}
Let $H$ be a closed subgroup of biholomorphisms of $D$. We will denote by $L^\infty(D)^H$ the subspace of $L^\infty(D)$ consisting of $H$-invariant symbols. In other words, we have
\[
	L^\infty(D)^H =
		\{a\in L^\infty(D) : h\cdot a = a, \text{ for all } h \in H\},
\]
where, for a given $a \in L^\infty(D)$ and $h \in H$, we define
\[
	(h \cdot a)(Z) = a(h^{-1}\cdot Z),
\]
for almost every $Z \in D$.

Symmetric pairs associated to symmetric domains can be used to obtain commutative $C^*$-algebras generated by Toeplitz operators by considering invariant symbols. The definitions and precise statements can be found in \cite{DOQJFA}. In this work, we will use the fact that the pairs $(\mathrm{Sp}(n,\mathbb{R}), \mathrm{GL}(n,\mathbb{R}))$ and $(\mathrm{Sp}(n,\mathbb{C}) \cap \mathrm{U}(n,n), \mathrm{U}(n))$ are symmetric in order to obtain the following consequence of \cite[Theorem~5.1]{DOQJFA}.

\begin{teo}[\cite{DOQJFA}]\label{teo:comm_GL_U}
	For every $\lambda > n$, the $C^*$-algebras $\mathcal{T}^{(\lambda)} (L^\infty(D^{III}_n)^{\mathrm{U}(n)})$ 
	and $\mathcal{T}^{(\lambda)} (L^\infty(\mathscr{S}_n)^{\mathrm{GL}(n,\mathbb{R})})$ acting on the weighted Bergman spaces $\mathcal{A}^2_\lambda(D^{III}_n)$ and $\mathcal{A}^2_\lambda(\mathscr{S}_n)$, respectively, are commutative.
\end{teo}

With the notation from subsection~\ref{subsec:ellhyperpar_actions}, Theorem~\ref{teo:comm_GL_U} states that for the Elliptic and Hyperbolic actions on $D^{III}_n$ and $\mathscr{S}_n$, respectively, the symbols invariant under such actions yield Toeplitz operators that generate commutative $C^*$-algebras.

The Parabolic Action provides the same sort of conclusion. This follows from the next consequence of \cite[Theorem~5.8]{DOQJFA}. We note that in this case the group $\mathrm{Symm}(n,\mathbb{R})$ does not yield a symmetric pair in the group $\mathrm{Sp}(n,\mathbb{R})$ of biholomorphisms of $\mathscr{S}_n$.

\begin{teo}[\cite{DOQJFA}]\label{teo:comm_Symm}
	For every $\lambda > n$, the $C^*$-algebra $\mathcal{T}^{(\lambda)} (L^\infty(\mathscr{S}_n)^{\mathrm{Symm}(n,\mathbb{R})})$ acting on the weighted Bergman space $\mathcal{A}^2_\lambda(\mathscr{S}_n)$ is commutative.
\end{teo}

\subsection{Moment map symbols}\label{subsec:momentmap-symbols}
Following \cite{QSJFAUnitBall,QRCAOT} we define the notion of moment map symbol for the setup of this work.

\begin{defi}\label{defi:momentmap_symbol}
	Let $D$ be either of the domains $D^{III}_n$ or $\mathscr{S}_n$ and $H$ a closed subgroup of the biholomorphism group of $D$. If $\mu^H : D \rightarrow \mathfrak{h}^*$ is a moment map for the action of $H$ on $D$, then a moment map symbol for $H$ or a $\mu^H$-symbol is a symbol $a \in L^\infty(D)$ that can be written in the form $a = f \circ \mu^H$ for some measurable function $f$. We denote by $L^\infty(D)^{\mu^H}$ the space of all essentially bounded measurable $\mu^H$-symbols on~$D$. 
\end{defi}

We have computed moment maps for the Abelian Elliptic, Abelian Hyperbolic and Parabolic actions in subsection~\ref{subsec:momentmaps_Abelian}. These computations allow us to provide the following simplified description of moment map symbols for these three actions. The claims are immediate consequences of Definition~\ref{defi:momentmap_symbol} and Propositions~\ref{propo:MomentMapT}, \ref{propo:MomentMapR+} and \ref{propo:MomentMapSymm}. 

\begin{propo}\label{propo:moment_symbols_explicit}
	Let $a \in L^\infty(D^{III}_n)$ and $b \in L^\infty(\mathscr{S}_n)$ be given. Then, the following equivalences hold
	\begin{enumerate}
		\item The measurable function $a$ is a $\mu^{\mathbb{T}}$-symbol if and only if there exists a measurable function $f$ such that $a(Z) = f\big(\mathrm{tr}\big((I_n - Z\overline{Z})^{-1}\big)\big)$, for almost every $Z \in D^{III}_n$.
		\item The measurable function $b$ is a $\mu^{\mathbb{R}_+}$-symbol if and only if there exists a measurable function $f$ such that $b(Z) = f\big(\mathrm{tr}\big(\mathrm{Im}(Z)^{-1} \mathrm{Re}(Z)\big)\big)$, for almost every $Z \in \mathscr{S}_n$.
		\item The measurable function $b$ is a $\mu^{\mathrm{Symm}(n,\mathbb{R})}$-symbol if and only if there exists a measurable function $f$ such that $b(Z) = f(\mathrm{Im}(Z))$, for almost every $Z \in \mathscr{S}_n$.
	\end{enumerate}
\end{propo}

By definition, the moment map symbols for Abelian groups are invariant under the corresponding actions. It turns out that the moment maps of the first two actions are in fact invariant under larger groups, those considered in Theorem~\ref{teo:comm_GL_U}. This is the content of the next two results.

\begin{propo}\label{propo:muT-U(n)inv}
	Let $\mu^{\mathbb{T}} : D^{III}_n \rightarrow \mathbb{R}$ be the moment map for the $\mathbb{T}$-action on $D^{III}_n$ given in Proposition~\ref{propo:MomentMapT}. Then, $\mu^\mathbb{T}$ is a $\mathrm{U}(n)$-invariant function. In particular, we have $L^\infty(D^{III}_n)^{\mu^{\mathbb{T}}} \subset L^\infty(D^{III}_n)^{\mathrm{U}(n)}$.
\end{propo}
\begin{proof}
	Recall that $\mathrm{U}(n)$ acts on $D^{III}_n$ by $U \cdot Z = UZU^T$. Using the expression of $\mu^{\mathbb{T}}$ obtained in Proposition~\ref{propo:MomentMapT}, we have for every $U \in \mathrm{U}(n)$
	\begin{align*}
		\mu^{\mathbb{T}} (U \cdot Z)
			&=-2\mathrm{tr}
				\big((I_n
					- UZU^T\overline{UZU^T)})^{-1}\big)  \\
			&=-2\mathrm{tr}
				\big((I_n
					- UZ\overline{Z}\overline{U^T})^{-1}\big)  \\
			&=-2\mathrm{tr}
				\big((U(I_n - Z\overline{Z})\overline{U^T})^{-1}
					\big) \\
			&=-2\mathrm{tr}
				\big(U(I_n-Z\overline{Z})^{-1}U^{-1}
					\big)  \\
			&=-2\mathrm{tr}
				\big((I_n-Z\overline{Z})^{-1}\big) 
				=\mu^{\mathbb{T}}(Z),
	\end{align*}
	for every $Z \in D^{III}_n$. The last claim is now an immediate consequence of Definition~\ref{defi:momentmap_symbol}.
\end{proof} 

\begin{propo}\label{propo:muR+-GL(n,R)inv}
	Let $\mu^{\mathbb{R}_+} : \mathscr{S}_n \rightarrow \mathbb{R}$ be the moment map of the $\mathbb{R}_+$-action on $\mathscr{S}_n$ given in Proposition~\ref{propo:MomentMapR+}. Then, $\mu^{\mathbb{R}_+}$ is a $\mathrm{GL}(n,\mathbb{R})$-invariant function. In particular, we have $L^\infty(\mathscr{S}_n)^{\mu^{\mathbb{R}_+}} \subset L^\infty(\mathscr{S}_n)^{\mathrm{GL}(n,\mathbb{R})}$.
\end{propo}
\begin{proof}
	Recall that $\mathrm{GL}(n,\mathbb{R})$ acts on $\mathscr{S}_n$ by $A\cdot Z = AZA^T$. We now use the expression of $\mu^{\mathbb{R}_+}$ obtained in Proposition~\ref{propo:MomentMapR+}, and for every $A \in \mathrm{GL}(n,\mathbb{R})$ we compute
	\begin{align*}
		\mu^{\mathbb{R}_+}(A\cdot Z)
			&= -4\mathrm{tr}
				\big((A\mathrm{Im}(Z) A^\top)^{-1} 
					A\mathrm{Re}(Z) A^\top
				\big)  \\
			&= -4\mathrm{tr}
				\big((A^\top)^{-1}\mathrm{Im}(Z)^{-1}A^{-1}
					A\mathrm{Re}(Z) A^\top
					\big)  \\
			&= -4\mathrm{tr}
				\big(\mathrm{Im}(Z)^{-1}
					\mathrm{Re}(Z)\big) 
				=\mu^{\mathbb{R}_+}(Z),
	\end{align*}
	for every $Z \in \mathscr{S}_n$. Again, the last claim now follows immediately. 
\end{proof}

For the Parabolic Action, it turns out that $\mathrm{Symm}(n,\mathbb{R})$-invariance and being a $\mu^{\mathrm{Symm}(n,\mathbb{R})}$-symbol are equivalent. This is the content of the next result.

\begin{propo}\label{propo:muSymm_and_invariance}
	For the moment map  $\mu^{\mathrm{Symm}(n,\mathbb{R})} : \mathscr{S}_n \rightarrow \mathrm{Symm}(n,\mathbb{R})$ of the $\mathrm{Symm}(n,\mathbb{R})$-action on $\mathscr{S}_n$ given in Proposition~\ref{propo:MomentMapSymm}, we have
	\[
		L^\infty(\mathscr{S}_n)^{\mu^{\mathrm{Symm}(n,\mathbb{R})}} 
		= L^\infty(\mathscr{S}_n)^{\mathrm{Symm}(n,\mathbb{R})}.
	\]
\end{propo}
\begin{proof}
	The $\mathrm{Symm}(n,\mathbb{R})$-action on $\mathscr{S}_n$ is given by the expression $S\cdot Z=Z+S$. Hence, for a given $a \in L^{\infty}(\mathscr{S}_n)$ we have the following sequence of equivalences 
	\begin{align*}
		a \text{ is }  
		&\mathrm{Symm}(n,\mathbb{R})\text{-invariant} \\
		&\Longleftrightarrow 
			\text{ for every } S \in \mathrm{Symm}(n,\mathbb{R}):\;
			a(Z+S)=a(Z) 
				\text{ for a.e. } Z \in \mathscr{S}_n  \\	
		&\Longleftrightarrow 
			\text{ for some measurable } f:\;
			a(Z)=f(\mathrm{Im}(Z)) 
				\text{ for a.e. } Z \in \mathscr{S}_n,
	\end{align*}
	and the result follows from the last case in Proposition~\ref{propo:moment_symbols_explicit}.
\end{proof}

\subsection{Commuting Toeplitz operators with moment maps symbols}
\label{subsec:comm_moment_map_symbols}
We now state one of our main results: for $D$ either of the domains $D^{III}_n$ or $\mathscr{S}_n$, there are three Abelian groups of biholomorphisms of $D$ to which we can associate commutative $C^*$-algebras generated by Toeplitz operators.

\begin{teo}\label{teo:commToep_3Abeliangroups}
	Let $D$ be either of the domains $D^{III}_n$ or $\mathscr{S}_n$. The Abelian Elliptic Action, the Abelian Hyperbolic Action and the Parabolic Action on $D$ yield three Abelian groups of biholomorphisms of $D$ which provide, for every $\lambda > n$, the following commutative $C^*$-algebras generated by Toeplitz operators.
	\begin{description}
		\item[Abelian Elliptic] The $C^*$-algebra $\mathcal{T}^{(\lambda)} \big(L^\infty(D^{III}_n)^{\mu^\mathbb{T}}\big)$, acting on $\mathcal{A}^2_\lambda(D^{III}_n)$, obtained from the moment map of the $\mathbb{T}$-action on $D^{III}_n$. 
		\item[Abelian Hyperbolic] The $C^*$-algebra $\mathcal{T}^{(\lambda)}
		\big(L^\infty(\mathscr{S}_n)^{\mu^{\mathbb{R}_+}}\big)$, acting on $\mathcal{A}^2_\lambda(\mathscr{S}_n)$, obtained from the moment map of the $\mathbb{R}_+$-action on $\mathscr{S}_n$.
		\item[Parabolic] The $C^*$-algebra $\mathcal{T}^{(\lambda)}
		\big(L^\infty(\mathscr{S}_n)^{\mu^{\mathrm{Symm}(n,\mathbb{R})}}
		\big)$, acting on $\mathcal{A}^2_\lambda(\mathscr{S}_n)$, obtained from the moment map of the $\mathrm{Symm}(n,\mathbb{R})$-action on $\mathscr{S}_n$.
	\end{description}
\end{teo}
\begin{proof}
	First, we note that Propositions~\ref{propo:muT-U(n)inv} and \ref{propo:muR+-GL(n,R)inv} imply the inclusions
	\begin{align*}
		\mathcal{T}^{(\lambda)}
			(L^\infty(D^{III}_n)^{\mu^\mathbb{T}})
			&\subset
			 \mathcal{T}^{(\lambda)}
			 	(L^\infty(D^{III}_n)^{\mathrm{U}(n)}) \\
		\mathcal{T}^{(\lambda)}
			(L^\infty(\mathscr{S}_n)^{\mu^{\mathbb{R}_+}})
			&\subset 
			\mathcal{T}^{(\lambda)}
			(L^\infty(\mathscr{S}_n)^{\mathrm{GL}(n,\mathbb{R})}),
	\end{align*}
	and so the cases of the Abelian Elliptic and Abelian Hyperbolic Actions follow from Theorem~\ref{teo:comm_GL_U}.
	
	For the Parabolic Action, we note that Proposition~\ref{propo:muSymm_and_invariance} yields the identity
	\[
		\mathcal{T}^{(\lambda)}
			(L^\infty(\mathscr{S}_n)^{\mu^{\mathrm{Symm}(n,\mathbb{R})}})
			= \mathcal{T}^{(\lambda)}
			(L^\infty(\mathscr{S}_n)^{\mathrm{Symm}(n,\mathbb{R})}),
	\]
	and now the result is a consequence of Theorem~\ref{teo:comm_Symm}.
\end{proof}

\begin{remark}\label{rmk:3Abelian_DIII_vs_Disk}
	It follows from the discussion in subsection~\ref{subsec:ellhyperpar_actions} (see Corollary~\ref{coro:3Actionn=1} and Definition~\ref{defi:AbelianEllHyp}) that for $n=1$ the three actions considered in Theorem~\ref{teo:commToep_3Abeliangroups} reduce to the usual elliptic, hyperbolic and parabolic actions known from complex analysis. These three actions have been previously used to obtain commutative $C^*$-algebras generated by Toeplitz operators, notably in the results found in~\cite{GKVElliptic,GKVHyperbolic,GKVParabolic,GQVJFA} (see also \cite{KorenblumZhu1995}). In fact, the commutative $C^*$-algebras generated by Toeplitz operators from Theorem~\ref{teo:commToep_3Abeliangroups} reduce to those from these previous works when $n = 1$.
	
	One of the main guiding lights in this line of study of Toeplitz operators has been to find generalizations to higher dimensions of these commutative $C^*$-algebras observed in the case of the unit disk. This was achieved for the unit ball $\mathbb{B}^n$ in $\mathbb{C}^n$ through the use of maximal Abelian subgroups of the biholomorphism group of $\mathbb{B}^n$ (see~\cite{QVBallI,QVBallII}). However, the result for the unit ball $\mathbb{B}^n$ from these references lead to $n+2$ different commutative $C^*$-algebras, which is in contrast with the simplicity of only three Abelian groups for the case of the unit disk. 
	
	On the other hand, our Theorem~\ref{teo:commToep_3Abeliangroups} recovers for higher dimensions the simplicity observed in the case of the unit disk. More precisely, we consider the generalized unit disk $D^{III}_n$ and its unbounded realization $\mathscr{S}_n$, Siegel's generalized upper half-plane. For these domains, Theorem~\ref{teo:commToep_3Abeliangroups} yields three commutative $C^*$-algebras generated by Toeplitz operators, acting on the Bergman spaces of $D^{III}_n$ and $\mathscr{S}_n$, which can be seen as natural extensions of the case of the unit disk. And this is achieved while using only three Abelian groups for any dimension. This is possible due to the fact that we have replaced invariant symbols with moment map symbols. This highlights the importance of using symplectic geometry to study Toeplitz operators acting on Bergman spaces in higher dimensions.	
\end{remark}

\section{Spectral integral formulas for Toeplitz operator with moment map symbols}
\label{sec:SpectralIntegrals}

In this final section we present explicit integral formulas that simultaneously diagonalize Toeplitz operators. This will be done for the Abelian Elliptic and Parabolic Actions.

\subsection{Toeplitz operators with Abelian Elliptic symbols}
In this case, we are dealing with symbols that belong to $\mathcal{T}^{(\lambda)} (L^\infty(D^{III}_n)^{\mu^\mathbb{T}})
\subset \mathcal{T}^{(\lambda)} (L^\infty(D^{III}_n)^{\mathrm{U}(n)})$.
In particular, it is useful to consider the $\mathrm{U}(n)$-action on the Bergman spaces over $D^{III}_n$. We recall some properties of such action and refer to \cite{Upmeier,DQRadial} for further details.

For every $\lambda > n$, there is a unitary representation given by 
\begin{align*}
	\pi_\lambda : \mathrm{U}(n) \times \mathcal{A}^2_\lambda(D^{III}_n) &\longrightarrow \mathcal{A}^2_\lambda(D^{III}_n) \\
	\pi_\lambda(A)(f) &= f \circ A^{-1}.
\end{align*}
This representation leaves invariant the subspace of (holomorphic) polynomials on $D^{III}_n \subset \mathrm{Symm}(n,\mathbb{C})$, that we will denote by $\mathcal{P}(\mathrm{Symm}(n,\mathbb{C})) = \mathcal{P}$, for simplicity. In particular, for every $\lambda > n$, the decomposition of $\mathcal{A}^2_\lambda(D^{III}_n)$ into irreducible $\mathrm{U}(n)$-submodules is the same as the one corresponding to the $\mathrm{U}(n)$-action on $\mathcal{P}$. Let us denote by $\overrightarrow{\mathbb{N}}^n$ the set of integer $n$-tuples that satisfy $\alpha_1 \geq \dots \geq \alpha_n \geq 0$. Then, using the representation $\pi_\lambda$, one can show that, for every $\lambda > n$, there is a Hilbert direct sum decomposition 
\begin{equation}\label{eq:PdecompPalpha}
	\mathcal{A}^2_\lambda(D^{III}_n) = 
		\bigoplus_{\alpha \in \overrightarrow{\mathbb{N}}^n}
			\mathcal{P}^\alpha,
\end{equation}
where $\big(\mathcal{P}^\alpha 
\big)_{\alpha \in \overrightarrow{\mathbb{N}}^n}$ is a family of mutually non-isomorphic $\mathrm{U}(n)$-submodules of~$\mathcal{P}$. For the proof of this claim we refer to \cite[Chapter~2]{Upmeier} (see also \cite{DQRadial,Johnson}). 

We consider the polynomials given by
\[
	\Delta_j(Z) = \det(Z_j)
\]
where $Z_j$ is the upper-left corner $j\times j$ submatrix of $Z$. For every $\alpha \in \overrightarrow{\mathbb{N}}^n$ we will also consider the polynomial
\[
	\Delta_\alpha(Z) = \prod_{j=1}^n 
			\Delta_j(Z)^{\alpha_j - \alpha_{j+1}}
\] 
where we agree to define $\alpha_{n+1} = 0$. These are known as the conical polynomials for the representation of $\mathrm{U}(n)$ on $\mathcal{P}$ (see \cite[Chapter~2]{Upmeier}).

With the previous notation, the following result is an application of Proposition~4.7 and Theorem~4.11 from \cite{DQRadial} to our current setup.

\begin{teo}[\cite{DQRadial}]\label{teo:Toeplitz_U(n)-invariant}
	Let $a \in L^\infty(D^{III}_n)^{\mathrm{U}(n)}$ and $\lambda > n$ be given. Then, the Toeplitz operator $T^{(\lambda)}_a$ acting on the Bergman space $\mathcal{A}^2_\lambda(D^{III}_n)$ preserves the Hilbert direct sum~\eqref{eq:PdecompPalpha}. Furthermore, we have
	\[
		T^{(\lambda)}_a|_{\mathcal{P}^\alpha} =
			c_{a,\lambda}(\alpha) I_{\mathcal{P}^\alpha},
	\]
	where the complex constant $c_{a,\lambda}(\alpha)$ is given by
	\[
		c_{a,\lambda}(\alpha) = 
		\frac{\displaystyle\int\limits_{0 < X < I_n} a(\sqrt{X}) 
				\Delta_\alpha(X) \Delta_n(I_n - X)^{\lambda-n-1} 
				\dif X }%
			{\displaystyle\int\limits_{0 < X < I_n} 
				\Delta_\alpha(X) \Delta_n(I_n - X)^{\lambda-n-1} 
				\dif X },
	\]
	for every $\alpha \in \overrightarrow{\mathbb{N}}^n$. The condition $0 < X < I_n$ denotes an open subset of $\mathrm{Symm}(n,\mathbb{R})$ and $\dif X$ the Lebesgue measure on the latter.
\end{teo}

Note that the integrals in Theorem~\ref{teo:Toeplitz_U(n)-invariant} are taken over an open subset of the vector space $\mathrm{Symm}(n,\mathbb{R})$, which is $n(n+1)/2$-dimensional. We will now simplify these expressions, and so the results of \cite{DQRadial}, to obtain integral formulas over lower dimensional spaces for Toeplitz operators with $\mathrm{U}(n)$-invariant symbols. Our goal is to reduce the number of variables over which the corresponding symbols have to be integrated. More precisely, we will obtain integral formulas involving the set
\[
	\overrightarrow{(0,1)}^n 
	= \{x \in (0,1)^n \mid x_n > \dots > x_1 > 0\},
\]
which is only $n$-dimensional. In the rest of this work, we will denote by $D(x)$ the diagonal $n \times n$ matrix with diagonal elements given by $x \in \mathbb{R}^n$. Also, for a given $x \in \mathbb{R}_+^n$ we will write $\sqrt{x} = (\sqrt{x}_1, \dots, \sqrt{x}_n)$.

\begin{teo}\label{teo:Toeplitz_U(n)-invariant-(0,1)n}
	Let $a \in L^\infty(D^{III}_n)^{\mathrm{U}(n)}$ and $\lambda > n$ be given. Then, the complex constants $(c_{a,\lambda}(\alpha))_{\alpha \in \overrightarrow{\mathbb{N}}^n}$ such that
	\[
		T^{(\lambda)}_a|_{\mathcal{P}^\alpha} =
			c_{a,\lambda}(\alpha) I_{\mathcal{P}^\alpha},
	\]
	for every $\alpha \in \overrightarrow{\mathbb{N}}^n$, as obtained in Theorem~\ref{teo:Toeplitz_U(n)-invariant}, are given by
	\[
		c_{a,\lambda}(\alpha) = 
		\frac{\displaystyle\int\limits_{\overrightarrow{(0,1)}^n}
			a(D(\sqrt{x})) H_{\alpha_n,\lambda}(x) h_\alpha(x) \dif x}%
			{\displaystyle\int\limits_{\overrightarrow{(0,1)}^n}
			H_{\alpha_n,\lambda}(x) h_\alpha(x) \dif x},
	\]
	where the functions $H_{\alpha_n,\lambda}, h_\alpha: \overrightarrow{(0,1)}^n \rightarrow [0,\infty)$ are defined by
	\begin{align*}
		H_{\alpha_n,\lambda}(x) &=
			\bigg(\prod_{j=1}^n x_j\bigg)^{\alpha_n}
			\bigg(\prod_{j=1}^n (1 - x_j)\bigg)^{\lambda - n - 1}
			\bigg(\prod_{j<k}(x_j - x_k)\bigg) \\
		h_\alpha(x) &= \int\limits_{\mathrm{O}(n)}
			\prod_{j=1}^{n-1} \Delta_j(A D(x) 
					A^\top)^{\alpha_j - \alpha_{j+1}} \dif A
	\end{align*}
	for every $\alpha \in \overrightarrow{\mathbb{N}}^n$, and $\dif A$ is a fixed Haar measure on $\mathrm{O}(n)$.
\end{teo}
\begin{proof}
	Let us denote
	\[
		I_{a,\lambda}(\alpha) = \int\limits_{0 < X < I_n} a(\sqrt{X}) 
			\Delta_\alpha(X) \Delta_n(I_n - X)^{\lambda-n-1} 
				\dif X.
	\]
	As noted above, this integral is taken over the open subset of  $\Omega_n$ that consists of the matrices $X$ satisfying $0 < X < I_n$. The linear maps that leave invariant $\Omega_n$ is realized by the action of the group $\mathrm{GL}(n,\mathbb{R})$ given by 
	\[
		A \cdot X = A X A^\top,
	\]
	where $A \in \mathrm{GL}(n,\mathbb{R})$ and $X \in \Omega_n$. In particular, the isotropy group of symmetries of the cone $\Omega_n$ that fixes $I_n$ is realized by the corresponding action of the group $\mathrm{O}(n)$. We refer to \cite[Section~1.3]{Upmeier} for the proof of these claims.
	
	By the previous discussion, it follows from \cite[Proposition~1.3.63]{Upmeier} that there is a constant $C > 0$ such that
	\begin{align*}
		I_{a,\lambda}(\alpha) = C \int\limits_{\overrightarrow{(0,1)}^n}
					&\int\limits_{\mathrm{O}(n)}
					a\Big(\sqrt{AD(x)A^\top}\Big)
						\Delta_\alpha\big(AD(x)A^\top\big) \times \\
					&\times\Delta_n
						\big(I_n - AD(x)A^\top\big)^{\lambda -n -1}
						\prod_{j<k}(x_j - x_k) 
						\dif A \dif x.
	\end{align*}
	Note that we have used the fact that the symmetric cone $\Omega_n$ has rank $n$ and characteristic multiplicity $a = 1$. We now observe that for every $A \in \mathrm{O}(n)$ and $x \in \overrightarrow{(0,1)}^n$ we have
	\[
		a\Big(\sqrt{AD(x)A^\top}\Big) 
			= a\big(AD(\sqrt{x})A^\top\big)
			= a(A\cdot D(\sqrt{x})) = a(D(\sqrt{x})),
	\]
	because $a$ is $\mathrm{U}(n)$-invariant. On the other hand, we have
	\begin{align*}
		\Delta_\alpha\big(AD(x)A^\top\big) 
			&= \prod_{j=1}^{n-1} \Delta_j\big(AD(x)A^\top
				\big)^{\alpha_j - \alpha_{j+1}}	
				\det\big(AD(x)A^\top\big)^{\alpha_n}	\\
			&= \prod_{j=1}^{n-1} \Delta_j\big(AD(x)A^\top 
				\big)^{\alpha_j - \alpha_{j+1}}	
				\det(D(x))^{\alpha_n} \\
			&= \prod_{j=1}^{n-1} \Delta_j\big(AD(x)A^\top 
				\big)^{\alpha_j - \alpha_{j+1}}	
				\bigg(\prod_{j=1}^n x_j\bigg)^{\alpha_n}	
	\end{align*}
	and a similar computation yields
	\begin{align*}
		\Delta_n\big(I_n - AD(x)A^\top\big)^{\lambda -n -1}
			&= \det(I_n - D(x))^{\lambda -n -1} \\
			&= \bigg(\prod_{j=1}^n (1 - x_j)\bigg)^{\lambda - n - 1},
	\end{align*}
	where the last two computations hold for every $A \in \mathrm{O}(n)$ and $x \in \overrightarrow{(0,1)}^n$ as~well. Collecting these identities we obtain
	\[
		I_{a,\lambda}(\alpha) = C \int\limits_{\overrightarrow{(0,1)}^n}
			a(D(\sqrt{x})) H_{\alpha_n,\lambda}(x) h_\alpha(x) 
					\dif x,
	\]
	where $H_{\alpha_n,\lambda}$ and $h_\alpha$ are given as in the statement. The result now follows once we observe that
	\[
		c_{a,\lambda}(\alpha) 
		= \frac{I_{a,\lambda}(\alpha)}{I_{1,\lambda}(\alpha)},
	\]
	for every $\alpha \in \overrightarrow{\mathbb{N}}^n$.
\end{proof}

The next goal in this subsection is to apply Theorem~\ref{teo:Toeplitz_U(n)-invariant-(0,1)n} to the case of the moment map symbols corresponding to the Abelian Elliptic Action. 
The description of such symbols provided by  Proposition~\ref{propo:moment_symbols_explicit} will greatly simplify our formulas.

\begin{teo}\label{teo:Toeplitz_muT}
	Let $a \in L^\infty(D^{III}_n)^{\mu^{\mathbb{T}}}$ and $\lambda > n$ be given. Let $f$ be a measurable function such that $a(Z) = f\big(\mathrm{tr}\big((I_n - Z\overline{Z})^{-1}\big)\big)$, for almost every $Z \in D^{III}_n$. Then, the complex constants $(c_{a,\lambda}(\alpha))_{\alpha \in \overrightarrow{\mathbb{N}}^n}$ such that
	\[
		T^{(\lambda)}_a|_{\mathcal{P}^\alpha} =
			c_{a,\lambda}(\alpha) I_{\mathcal{P}^\alpha},
	\]
	for every $\alpha \in \overrightarrow{\mathbb{N}}^n$, as obtained in Theorem~\ref{teo:Toeplitz_U(n)-invariant}, are given by
	\[
		c_{a,\lambda}(\alpha) = 
		\frac{\displaystyle\int\limits_{\overrightarrow{(0,1)}^n}
			f\bigg(\sum_{j=1}^n \frac{1}{1-x_j}\bigg)
			H_{\alpha_n,\lambda}(x) h_\alpha(x) \dif x}%
		{\displaystyle\int\limits_{\overrightarrow{(0,1)}^n}
			H_{\alpha_n,\lambda}(x) h_\alpha(x) \dif x},
	\]
	where the functions $H_{\alpha_n,\lambda}, h_\alpha : \overrightarrow{(0,1)}^n \rightarrow [0,\infty)$ are those defined in Theorem~\ref{teo:Toeplitz_U(n)-invariant-(0,1)n}.
\end{teo}
\begin{proof}
	It is an immediate consequence of Theorem~\ref{teo:Toeplitz_U(n)-invariant-(0,1)n} and the computation
	\begin{align*}
		a(D(\sqrt{x})) 
			&= f\big(\mathrm{tr}
				\big((I_n - D(\sqrt{x})\overline{D(\sqrt{x})})^{-1}\big)\big) \\
			&= f\big(\mathrm{tr}\big((I_n - D(x))^{-1}\big)\big) \\
			&= f\bigg(\sum_{j=1}^n \frac{1}{1-x_j}\bigg),
	\end{align*}
	which holds for every $x \in \overrightarrow{(0,1)}^n$.
\end{proof}

\subsection{Toeplitz operators with Parabolic symbols}
We recall from subsection~\ref{subsec:BergmanToeplitz} the decomposition
\[
	\mathscr{S}_n = \mathrm{Symm}(n,\mathbb{R}) \oplus i \Omega_n,
\]
where $\Omega_n = \mathrm{Pos}(n,\mathbb{R})$. With respect to this decomposition, and for every $\lambda > n$, the weighted measure $\widehat{v}_\lambda$ decomposes as
\[
	\dif \widehat{v}_\lambda(Z) =
		C_{\lambda,n} \det(2Y)^{\lambda-n-1} \dif X \dif Y,
\]
with the coordinates $Z = X + i Y$, ($X \in \mathrm{Symm}(n,\mathbb{R})$ and $Y \in \Omega_n$) and for the positive constant
\[
	C_{\lambda, n} = 
		\frac{\Gamma_{\Omega_n}(\lambda)}%
			{\pi^{\frac{n(n+1)}{2}} \Gamma_{\Omega_n}\big(\lambda - \frac{n+1}{2}\big)}.
\]
This yields the natural isometry
\[
	L^2(\mathscr{S}_n,\widehat{v})
		\simeq L^2(\mathrm{Symm}(n,\mathbb{R}), \dif X) \otimes 
			L^2(\Omega_n, C_{\lambda,n} \det(2Y)^{\lambda-n-1} \dif Y),
\]
that we will use in the rest of this work.

Let us consider the unitary operator $U = \mathcal{F} \otimes I$ defined on $L^2(\mathscr{S}_n, \widehat{v}_\lambda)$, where $\mathcal{F}$ is the Fourier transform on $\mathrm{Symm}(n,\mathbb{R})$. More precisely, we have
\[
	(\mathcal{F}(f))(X)
		= \frac{1}{(2\pi)^{\frac{n(n+1)}{4}}} 
		\int_{\mathrm{Symm}(n,\mathbb{R})}
			e^{-i\mathrm{tr}(X\xi)} f(\xi) \dif \xi
\]
for every $f \in L^1(\mathrm{Symm}(n,\mathbb{R})) \cap L^2(\mathrm{Symm}(n,\mathbb{R}))$. In particular, we use as canonical inner product on $\mathrm{Symm}(n,\mathbb{R})$ the one induced by the trace. We recall that, with respect to such inner product, the cone $\Omega_n$ is self-dual in the sense that
\[
	\Omega_n = \{ \xi \in \mathrm{Symm}(n,\mathbb{R}) \mid
		\mathrm{tr}(\xi X) > 0 \text{ for all } 
			X \in \overline{\Omega}_n \setminus \{0\} \}.
\]
We will use this fundamental property (see \cite{Upmeier}) to apply some well known formulas associated to symmetric cones.

The next two results allow to describe the Bergman spaces after applying the unitary map $U$.

\begin{lema}\label{lem:Slambda}
	Let $\mathcal{H}_\lambda(\mathscr{S}_n) = U(\mathcal{A}^2_\lambda(\mathscr{S}_n))$ be the image of the Bergman space $\mathcal{A}^2_\lambda(\mathscr{S}_n)$ under the unitary map $U$. Then, the operator given by
	\begin{align*}
		S_\lambda : L^2(\Omega_n) &\longrightarrow
			L^2(\mathrm{Symm}(n,\mathbb{R})) \otimes
				L^2(\Omega_n, C_{\lambda,n} \det(2Y)^{\lambda-n-1} 
					\dif Y) \\
		(S_\lambda(f))(X,Y) &= 
			\frac{(2\pi)^{\frac{n(n+1)}{4}}}%
			{\Gamma_{\Omega_n}(\lambda)^{\frac{1}{2}}}
			\chi_{\Omega_n}(X) f(X) \det(X)^{\frac{\lambda}{2}-\frac{n+1}{4}}
			e^{-\mathrm{tr}(XY)},
	\end{align*}
	is an isometry onto $\mathcal{H}_\lambda(\mathscr{S}_n)$.
\end{lema}
\begin{proof}
	From the basic properties of the Fourier transform, the Cauchy-Riemann equations on $\mathscr{S}_n$ are transformed under $U$ to the equations
	\[
		\bigg(X_{jk} + \frac{\partial}{\partial Y_{jk}}\bigg) 
			\varphi=0,
	\]
	which must hold for every $1 \leq j \leq k \leq n$. The general solution of these equations is $\varphi(X,Y) = \psi(X)e^{-\mathrm{tr}(XY)}$. Next, we need to consider the $L^2$-integrability of these solutions, and for this we evaluate
	\begin{align*}
		\int_{\mathscr{S}_n}|
			&\varphi(X,Y)|^2 C_{\lambda,n} \det(2Y)^{\lambda-n-1} 
				\dif X \dif Y =  \\
			=\;&C_{\lambda,n} \int_{\mathscr{S}_n} 
			|\psi(X)|^2 e^{-2\mathrm{tr}(XY)}
				\det(2Y)^{\lambda-n-1} \dif X \dif Y  \\
			=\;& C_{\lambda,n} 			
				\int_{\mathrm{Symm}(n,\mathbb{R})} |\psi(X)|^2
				\bigg(
				\int_{\Omega_n} 	
					e^{-2\mathrm{tr}(XY)} \det(2Y)^{\lambda-n-1}
					\dif Y
				\bigg) \dif X.
	\end{align*}
	For this to be finite it is necessary that $\mathrm{supp}(\psi) \subset \Omega_n$. On the other hand, \cite[Equation~2.4.30]{Upmeier} implies that (after some simple changes of variable) we have
	\[
		\int_{\Omega_n} 	
			e^{-2\mathrm{tr}(XY)} \det(2Y)^{\lambda-n-1}
				\dif Y =
			\frac{\Gamma_{\Omega_n}\big(\lambda - \frac{n+1}{2}\big)}{2^{\frac{n(n+1)}{2}}}
				\det(X)^{\frac{n+1}{2} - \lambda}
	\]
	Hence, in the above solution of the Cauchy-Riemann equations we replace $\psi(X)$ by the function
	\[
		\psi(X) = 
		\frac{(2\pi)^{\frac{n(n+1)}{4}}}%
			{\Gamma_{\Omega_n}(\lambda)^{\frac{1}{2}}}
			\chi_{\Omega_n}(X) f(X)
			\det(X)^{\frac{\lambda}{2}-\frac{n+1}{4}}
	\]
	for a suitable function $f$. With these choices and the previous computations we obtain
	\begin{align*}
		\|\varphi\|^2_{\mathcal{H}_\lambda(\mathscr{S}_n)} 
			=\;& C_{\lambda,n}
				\int_{\Omega_n}
				\frac{(2\pi)^{\frac{n(n+1)}{2}}}%
				{\Gamma_{\Omega_n}(\lambda)}
				|f(X)|^2
				\det(X)^{\lambda-\frac{n+1}{2}} \times \\
				& \times \frac{\Gamma_{\Omega_n}\big(\lambda - \frac{n+1}{2}\big)}{2^{\frac{n(n+1)}{2}}}
				\det(X)^{\frac{n+1}{2} - \lambda}
				\dif X \\
			=\;& C_{\lambda,n} 
				\frac{\pi^{\frac{n(n+1)}{2}} \Gamma_{\Omega_n}\big(\lambda - \frac{n+1}{2}\big)}%
				{\Gamma_{\Omega_n}(\lambda)} \|f\|^2_{L^2(\Omega_n)}
				= \|f\|^2_{L^2(\Omega_n)},
	\end{align*}
	where we have used the definition of the constant $C_{\lambda,n}$. The last set of identities completes the proof by the definition of $S_\lambda$.
\end{proof}

\begin{lema}\label{lem:SlambdaAdjoint}
	The adjoint operator of $S_\lambda$ from Lemma~\ref{lem:Slambda} is a partial isometry with initial space $\mathcal{H}_\lambda(\mathscr{S}_n)$ and final space $L^2(\Omega_n)$. Furthermore, we have
	\begin{align*}
	(S_\lambda^*(\varphi))(X) =\;&
		\frac{2^{\frac{n(n+1)}{4}}\Gamma_{\Omega_n} (\lambda)^{\frac{1}{2}}}%
			{\pi^{\frac{n(n+1)}{4}}\Gamma_{\Omega_n}
				\big(\lambda-\frac{n+1}{2}\big)}
		\det(X)^{\frac{\lambda}{2}-\frac{n+1}{4}} \times \\
		& \times \int_{\Omega_n}
			\varphi(X,Y)
			e^{-\mathrm{tr}(XY)}
				\det(2Y)^{\lambda-n-1} \dif Y,
	\end{align*}
	for every $\varphi \in L^2(\mathrm{Symm}(n,\mathbb{R})) \otimes
	L^2(\Omega_n, C_{\lambda,n} \det(2Y)^{\lambda-n-1} \dif Y)$.
\end{lema}
\begin{proof}
	The first claim follows from Lemma~\ref{lem:Slambda}. On the other hand, the expression for $S_\lambda^*$ is a consequence of the following straightforward computation for $f$ and $\varphi$ in the corresponding spaces
	\begin{align*}
		\langle S_\lambda(f), &\varphi \rangle = \\
		=\;& C_{\lambda,n} \int_{\mathscr{S}_n}
			(S_\lambda(f))(X,Y) \overline{\varphi(X,Y)}
				\det(2Y)^{\lambda-n-1} \dif X \dif Y \\
		=\;& C_{\lambda,n} 
			\frac{(2\pi)^{\frac{n(n+1)}{4}}}%
				{\Gamma_{\Omega_n}(\lambda)^{\frac{1}{2}}}
			\int_{\Omega_n} f(X) 	
				 \times \\
		&\times 
			\det(X)^{\frac{\lambda}{2}-\frac{n+1}{4}}
			\overline{
			\bigg(\int_{\Omega_n}
				\varphi(X,Y)
				e^{-\mathrm{tr}(XY)} 
				\det(2Y)^{\lambda-n-1}				
			\bigg)
			} \dif X,
	\end{align*}
	where we have used again the value of $C_{\lambda,n}$.
\end{proof}

The next result provides a formula for the Bergman projection after applying the unitary map $U$.

\begin{lema}\label{lem:BergmanU}
	Let $B_\lambda = U B_{\mathscr{S}_n,\lambda} U^*$ be the orthogonal projection
	\[
		L^2(\mathrm{Symm}(n,\mathbb{R})) \otimes	
			L^2(\Omega_n, C_{\lambda,n} \det(2Y)^{\lambda-n-1} \dif Y)
			\longrightarrow
			 \mathcal{H}_\lambda(\mathscr{S}_n).
	\]
	Then, we have the identities
	\[
		S_\lambda^* S_\lambda = I_{L^2(\Omega_n)}, \quad
		S_\lambda S_\lambda^* = B_\lambda.
	\]
	In particular, the orthogonal projection $B_\lambda$ is given by 	\begin{align*}
	(B_\lambda(\varphi))(X,Y) =\;& 
		\frac{2^\frac{n(n+1)}{2}}{\Gamma_{\Omega_n}
				\big(\lambda-\frac{n+1}{2}\big)}
				\chi_{\Omega_n}(X)
		 	\det(X)^{\lambda-\frac{n+1}{2}}
		 	e^{-\mathrm{tr}(XY)} \times   \\
 			&\times
			\int_{\Omega_n}
		    e^{-\mathrm{tr}(X\eta)}
		    \det(2\eta)^{\lambda-n-1}\varphi(X,\eta) \dif \eta.
	\end{align*}
\end{lema}
\begin{proof}
	By Lemma~\ref{lem:Slambda}, the operator $S_\lambda$ is a partial isometry with initial space $L^2(\Omega_n)$ and final space $\mathcal{H}_\lambda(\mathscr{S}_n)$. This implies the first two identities in the statement. Hence, it remains to compute $S_\lambda S_\lambda^*$ explicitly, and this is done as follows for every $\varphi \in \mathcal{H}_\lambda(\mathscr{S}_n)$
	\begin{align*}
		((S_\lambda S_\lambda^*)(\varphi))(X,Y) &= \\
			= \frac{(2\pi)^{\frac{n(n+1)}{4}}}%
			{\Gamma_{\Omega_n}(\lambda)^{\frac{1}{2}}} &
			\chi_{\Omega_n}(X) (S_\lambda^*(\varphi))(X) \det(X)^{\frac{\lambda}{2}-\frac{n+1}{4}}
			e^{-\mathrm{tr}(XY)} \\
			= \frac{(2\pi)^{\frac{n(n+1)}{4}}}%
			{\Gamma_{\Omega_n}(\lambda)^{\frac{1}{2}}} &
			\chi_{\Omega_n}(X)
			\bigg(
				\frac{2^{\frac{n(n+1)}{4}}\Gamma_{\Omega_n} 
					(\lambda)^{\frac{1}{2}}}%
				{\pi^{\frac{n(n+1)}{4}}\Gamma_{\Omega_n}
					\big(\lambda-\frac{n+1}{2}\big)}
				\det(X)^{\frac{\lambda}{2}-\frac{n+1}{4}} \times \\
			& \times \int_{\Omega_n}
				\varphi(X,\eta)
				e^{-\mathrm{tr}(X\eta)}
				\det(2\eta)^{\lambda-n-1} \dif \eta 
			\bigg),
	\end{align*}
	which clearly simplifies to the required expression.
\end{proof}

The constructions considered so far allow us to introduce in the next result a Fourier-Laplace transform from $\mathcal{A}^2_\lambda(\mathscr{S}_n)$ onto $L^2(\Omega_n)$. We refer to \cite[Proposition~2.4.26]{Upmeier} for a similar related construction.

\begin{teo}\label{teo:RlambdaFourierLaplace}
	With the current notation and for every $\lambda>n$, the operator $R_\lambda = S_\lambda^* U : L^2_\lambda(\mathscr{S}_n,\widehat{v}_\lambda ) \rightarrow  L^2(\Omega_n)$ is a partial isometry with initial space $\mathcal{A}^2_\lambda(\mathscr{S}_n)$ and final space $L^2(\Omega_n)$. In particular, its adjoint 
	\[
		R_\lambda^* : L^2(\Omega_n) \longrightarrow
				L^2(\mathscr{S}_n,\widehat{v}_\lambda)
	\]
	is an isometry onto $\mathcal{A}^2_\lambda(\mathscr{S}_n)$. Furthermore, we have
	\[
		(R_\lambda^*(f))(Z)
			= \frac{1}{\Gamma_\Lambda(\lambda)^{\frac{1}{2}}}
			\int_{\Omega_n}
			f(\xi) \det(\xi)^{\frac{\lambda}{2}-\frac{n+1}{4}}
			e^{i\mathrm{tr}(\xi Z)} \dif \xi,
	\]
	for every $f \in L^2(\Omega_n)$ and $Z \in \mathscr{S}_n$.
\end{teo}

\begin{proof}
	Since $U$ is a unitary operator mapping $\mathcal{A}^2_\lambda(\mathscr{S}_n)$ onto $\mathcal{H}_\lambda(\mathscr{S}_n)$, it follows from Lemma~\ref{lem:SlambdaAdjoint} that $R_\lambda$ is a partial isometry with the indicated initial and final spaces. From this we now conclude that $R_\lambda^*$ is an isometry from $L^2(\Omega_n)$ onto $\mathcal{A}^2_\lambda(\mathscr{S}_n)$.

	It only remains to find the expression stated for $R_\lambda^*$, which is achieved in the following computation. For every $f \in L^2(\Omega_n)$ and $Z \in \mathscr{S}_n$ we have
	\begin{align*}
	(R_\lambda^*(f))(Z)
		=\;& ((U^* S_\lambda)(f))(Z)
		= ((\mathcal{F}^{-1} \otimes I)\circ S_\lambda(f))(Z)  \\
		=\;& (\mathcal{F}^{-1}\otimes I)
		\bigg(
			\frac{(2\pi)^{\frac{n(n+1)}{4}}}%
				{\Gamma_{\Omega_n}(\lambda)^{\frac{1}{2}}}
				\chi_{\Omega_n}(X) f(X)
				\det(X)^{\frac{\lambda}{2}-\frac{n+1}{4}}
				e^{-\mathrm{tr}(XY)}
		\bigg) \\
		=\;& \frac{1}{\Gamma_{\Omega_n}(\lambda)^{\frac{1}{2}}}
			\int_{\Omega_n} f(\xi)
			\det(\xi)^{\frac{\lambda}{2}-\frac{n+1}{4}}
			e^{-\mathrm{tr}(\xi Y)} 
			e^{i\mathrm{tr}(X\xi)}
			\dif \xi \\
		=\;&\frac{1}{\Gamma_{\Omega_n}(\lambda)^{\frac{1}{2}}}
			\int_{\Omega_n}
			f(\xi) \det(\xi)^{\frac{\lambda}{2}-\frac{n+1}{4}}
			e^{i\mathrm{tr}(\xi Z)}d\xi,
	\end{align*}
	where $Z = X + i Y$, with $X, Y$ real matrices.
\end{proof}

We recall from Proposition~\ref{propo:muSymm_and_invariance} that
\[
	L^\infty(\mathscr{S}_n)^{\mu^{\mathrm{Symm}(n,\mathbb{R})}} 
		= L^\infty(\mathscr{S}_n)^{\mathrm{Symm}(n,\mathbb{R})}.
\]
In other words, the $\mathrm{Symm}(n,\mathbb{R})$-invariant symbols and the moment map symbols for the $\mathrm{Symm}(n,\mathbb{R})$-action on $\mathscr{S}_n$ are the same. Hence, the next result provides integral formulas that simultaneously diagonalizes Toeplitz operators with either type of symbols. 

\begin{teo}\label{teo:Toeplitz_SymmnR}
	Let $a \in L^\infty(\mathscr{S}_n)$ be a $\mathrm{Symm}(n,\mathbb{R})$-invariant symbol and $\lambda > n$ be given. Then, for $R_\lambda$ the operator from Theorem~\ref{teo:RlambdaFourierLaplace} we have a commutative diagram
	\[
		\xymatrix{
		\mathcal{A}^2_\lambda(\mathscr{S}_n) 
		\ar[d]_{T^{(\lambda)}_a} \ar[r]^<(.2){R_\lambda} &
			L^2(\Omega_n)
			\ar[d]^{M_{\gamma_{a,\lambda}}} \\
		\mathcal{A}^2_\lambda(\mathscr{S}_n) 
		\ar[r]^<(.2){R_\lambda} &
			L^2(\Omega_n) \\
		}
	\]
	where $\gamma_{a,\lambda} \in L^\infty(\Omega_n)$ is  given by
	\begin{align*}
		\gamma_{a,\lambda}&(X) = \\
			&= \frac{2^{\frac{n(n+1)}{2}} \det(X)^{\lambda-\frac{n+1}{2}}}%
			{\Gamma_{\Omega_n}
				\big(\lambda-\frac{n+1}{2}\big)}
			\int_{\Omega_n}
				a(Y) e^{-2\mathrm{tr}(XY)} 
				\det(2Y)^{\lambda-n-1} \dif Y,
	\end{align*}
	for every $X \in \Omega_n$.
\end{teo}
\begin{proof}
	For our given $\mathrm{Symm}(n,\mathbb{R})$-invariant symbol $a \in L^\infty(\mathscr{S}_n)$ we have $T_a^{(\lambda)} = B_{\mathscr{S}_n,\lambda}\circ M_a$. Then, Theorem~\ref{teo:RlambdaFourierLaplace} implies that
	\begin{align*}
		R_\lambda T_a^{(\lambda)} R_\lambda^*
		&= R_\lambda 
			B_{\mathscr{S}_n,\lambda} M_a B_{\mathscr{S}_n,\lambda} 
			R_\lambda^*  \\
		&= R_\lambda (R_\lambda^* R_\lambda) M_a 
			(R_\lambda^*R_\lambda )R_\lambda^*\\
		&= R_\lambda M_a R_\lambda^*
		= S_\lambda^* U M_a U^* S_\lambda \\
		&= S_\lambda^* M_a S_\lambda,
	\end{align*}
	where we have used that $U M_a U^* = M_a$, since 
	$U = \mathcal{F} \otimes I$ and $a$ depends only on $Y = \mathrm{Im}(Z)$.
	
	We now evaluate the last composition as follows for every $f \in L^2(\Omega_n)$ and $X \in \Omega_n$
	\begin{align*}
		(S_\lambda^* &M_a S_\lambda(f))(X) = \\
			=\;& 
				\frac{2^{\frac{n(n+1)}{4}} \Gamma_{\Omega_n} (\lambda)^{\frac{1}{2}}}%
				{\pi^{\frac{n(n+1)}{4}} \Gamma_{\Omega_n}
				\big(\lambda-\frac{n+1}{2}\big)}
			\det(X)^{\frac{\lambda}{2}-\frac{n+1}{4}} \times \\
			& \int_{\Omega_n}
				\bigg(a(Y)
				\frac{(2\pi)^{\frac{n(n+1)}{4}}}%
				{\Gamma_{\Omega_n} (\lambda)^{\frac{1}{2}}}
				f(X) \det(X)^{\frac{\lambda}{2}-\frac{n+1}{4}}
				e^{-\mathrm{tr}(XY)} 
				\bigg) \times \\
			&\qquad\qquad e^{-\mathrm{tr}(XY)}
				\det(2Y)^{\lambda-n-1} \dif Y \\
			=\;& 
				\frac{2^{\frac{n(n+1)}{2}} \det(X)^{\lambda-\frac{n+1}{2}}}%
				{\Gamma_{\Omega_n}
				\big(\lambda-\frac{n+1}{2}\big)} f(X)  
				\int_{\Omega_n}
				a(Y) e^{-2\mathrm{tr}(XY)} 
				\det(2Y)^{\lambda-n-1} \dif Y,
	\end{align*}
	which yields the required conclusion.
\end{proof}

\begin{remark}\label{rmk:ConesVasilevski}
	Theorem~\ref{teo:Toeplitz_SymmnR} can be seen as a generalization of some of the results found in \cite{VasilevskiTube}. More precisely, \cite[Theorem~4.1]{VasilevskiTube} provides the diagonalization of Toeplitz operators with the so-called cone component symbols, and such results holds for every tubular domain. However, \cite{VasilevskiTube} considers only the weightless case. On the other hand, we have considered only the tubular domain $\mathscr{S}_n$, but our result holds for arbitrarily weighted Bergman spaces and Toeplitz operators with symbols that depend only on the cone coordinates.
\end{remark}

\subsection*{Acknowledgment}
This research was supported by a Conacyt scholarship, SNI-Conacyt and Conacyt grants 280732 and 61517.


\begin{thebibliography}{XX}
	\bibitem{DOQJFA} Dawson, Matthew; \'Olafsson, Gestur and Quiroga-Barranco, Raul: \emph{Commuting   Toeplitz   operators   on   bounded symmetric   domains   and   multiplicity-free restrictions   of   holomorphic   discrete   series}. J. Funct. Anal. 268 (2015), no. 7, 1711--1732.
	
	\bibitem{DQRadial} Dawson, Matthew and Quiroga-Barranco, Raul: \emph{Radial Toeplitz operators on the weighted Bergman spaces of Cartan domains}. Representation theory and harmonic analysis on symmetric spaces, 97--114, Contemp. Math., 714, Amer. Math. Soc., Providence, RI, 2018.
	
	\bibitem{GKVElliptic} Grudsky, S., Karapetyants, A. and Vasilevski, N.: \emph{Toeplitz operators on the unit ball in $\mathbb{C}^n$ with radial symbols}. J. Operator Theory 49 (2003), no. 2, 325--346.
	
	\bibitem{GKVHyperbolic} Grudsky, S., Karapetyants, A. and Vasilevski, N.: \emph{Dynamics of properties of Toeplitz operators on the upper half-plane: hyperbolic case}. Bol. Soc. Mat. Mexicana (3) 10 (2004), no. 1, 119--138.
	
	\bibitem{GKVParabolic} Grudsky, S., Karapetyants, A. and Vasilevski, N.: \emph{Dynamics of properties of Toeplitz operators on the upper half-plane: parabolic case}. J. Operator Theory 52 (2004), no. 1, 185--214.
	
	\bibitem{GKVRadial} Grudsky, S., Karapetyants, A. and Vasilevski, N.: \emph{Dynamics of properties of Toeplitz operators with radial symbols}. Integral Equations Operator Theory 50 (2004), no. 2, 217--253.
	
	\bibitem{GQVJFA} Grudsky, S., Quiroga-Barranco, R. and Vasilevski N.: \emph{Commutative $C^*$-algebras of Toeplitz operators and quantization on the unit disk}. J. Funct. Anal. 234 (2006), no. 1, 1--44.
	
	\bibitem{Johnson} Johnson, Kenneth D.: \emph{On a ring of invariant polynomials on a Hermitian symmetric space}. J. Algebra 67 (1980), no. 1, 72--81. 
	
	\bibitem{Helgason} Helgason, Sigurdur: Differential geometry, Lie groups, and symmetric spaces. Corrected reprint of the 1978 original. Graduate Studies in Mathematics, 34. American Mathematical Society, Providence, RI, 2001.
	
	\bibitem{Hua} Hua, L. K.: Harmonic analysis of functions of several complex variables in the classical domains. Translated from the Russian by Leo Ebner and Adam Kor\'anyi American Mathematical Society, Providence, R.I. 1963.
	
	\bibitem{KorenblumZhu1995} Korenblum, Boris and Zhu, Ke He: \emph{An application of Tauberian theorems to Toeplitz operators}. J. Operator Theory 33 (1995), no. 2, 353--361.
	
	\bibitem{McDuffSalamon} McDuff, Dusa and Salamon, Dietmar: Introduction to symplectic topology. Third edition. Oxford Graduate Texts in Mathematics. Oxford University Press, Oxford, 2017.
	
	\bibitem{Mok} Mok, Ngaiming: Metric rigidity theorems on Hermitian locally symmetric manifolds. Series in Pure Mathematics, 6. World Scientific Publishing Co., Inc., Teaneck, NJ, 1989.

	\bibitem{QSJFAUnitBall} Quiroga-Barranco, Raul and Sanchez-Nungaray, Armando: \emph{Moment maps of Abelian groups and commuting Toeplitz operators acting on the unit ball}, Journal of Functional Analysis 281 (2021), no. 3, article 109039.

	\bibitem{QRCAOT} Quiroga-Barranco, Raul and Seng, Monyrattanak: \emph{Commuting Toeplitz operators on Cartan domains of type IV and moment maps}. Complex Anal. Oper. Theory 16 (2022), no. 7, Paper No. 102, 41 pp. 
	
	\bibitem{QVBallI} Quiroga-Barranco, Raul and Vasilevski, Nikolai: \emph{Commutative $C^*$-algebras of Toeplitz operators on the unit ball. I. Bargmann-type transforms and spectral representations of Toeplitz operators}. Integral Equations Operator Theory 59 (2007), no. 3, 379--419.
	
	\bibitem{QVBallII} Quiroga-Barranco, Raul and Vasilevski, Nikolai: \emph{Commutative $C^*$-algebras of Toeplitz operators on the unit ball. II. Geometry of the level sets of symbols}. Integral Equations Operator Theory 60 (2008), no. 1, 89--132.

	\bibitem{Range} Range, R. Michael: Holomorphic functions and integral representations in several complex variables. Graduate Texts in Mathematics, 108. Springer-Verlag, New York, 1986.
	
	\bibitem{Upmeier} Upmeier, Harald: Toeplitz operators and index theory in several complex variables. Operator Theory: Advances and Applications, 81. Birkh\"auser Verlag, Basel, 1996.
	
	\bibitem{VasilevskiTube} Vasilevski, N. L.: \emph{Bergman space on tube domains and commuting Toeplitz operators}. Proceedings of the Second ISAAC Congress, Vol. 2 (Fukuoka, 1999), 1523--1537, Int. Soc. Anal. Appl. Comput., 8, Kluwer Acad. Publ., Dordrecht, 2000. 
\end{thebibliography}
\end{document}